\documentclass{article}

\usepackage{epsfig}

\usepackage{pstricks,pstricks-add,pst-math,pst-xkey}
 \usepackage{amssymb,amsmath, amsthm}
\usepackage[T1]{fontenc}
\usepackage[latin1]{inputenc}
\usepackage[english]{babel}
\usepackage{amsfonts,amsbsy}
\usepackage{enumerate}
\usepackage{mathrsfs}
\usepackage{graphicx,psfrag}
\usepackage{epic}
\usepackage{stmaryrd}
\usepackage{longtable}
\allowdisplaybreaks[1]

\usepackage[backref=page,bookmarksopen=true,dvipdfm]%,colorlinks]
{hyperref}

 \usepackage{fancyhdr}
\theoremstyle{plain}
\newtheorem{proposition}{Proposition}[section]
\newtheorem{theorem}[proposition]{Theorem}
\newtheorem{lemma}[proposition]{Lemma}
\newtheorem{corollary}[proposition]{Corollary}

\theoremstyle{definition}

\theoremstyle{remark}

\newtheorem{remark}[proposition]{Remark}

\renewenvironment{proof}{\smallskip\noindent\emph{\textbf{Proof.}}\hspace{1pt}}%
{\hspace{-5pt}{\nobreak\quad\nobreak\hfill\nobreak$\square$\vspace{8pt}%
\par}\smallskip\goodbreak}

\newenvironment{proofof}[1]{\smallskip\noindent\emph{\textbf{Proof of #1.}}%
\hspace{1pt}}{\hspace{-5pt}{\nobreak\quad\nobreak\hfill\nobreak%
$\square$\vspace{8pt}\par}\smallskip\goodbreak}

\renewcommand{\leq}{\leqslant}
\renewcommand{\geq}{\geqslant}
\newcommand{\Div}{{\mathrm{Div}}}

\newcommand{\C}[1]{\mathscr{C}^{#1}}
\newcommand{\Cc}[1]{\mathscr{C}_c^{#1}}
\newcommand{\modulo}[1]{{\left|#1\right|}}
\newcommand{\norma}[1]{{\left\|#1\right\|}}
\newcommand{\Ref}[1]{{\rm(\ref{#1})}}
\newcommand{\reali}{{\mathbb{R}}}

\newcommand{\rpic}{{\mathbb{R}_+}}
\newcommand{\rpis}{{\mathbb{R}_+^*}}

\newcommand{\BV}{\mathbf{BV}}

\renewcommand{\epsilon}{\varepsilon}
\renewcommand{\phi}{\varphi}
\renewcommand{\L}[1]{{\mathbf{L}^#1}}

\newcommand{\W}[2]{{\mathbf{W}^{#1,#2}}}

\newcommand{\Lloc}[1]{{\mathbf{L}_{loc}^{#1}}}
\newcommand{\tv}{\mathinner{\rm TV}}

\renewcommand{\div}{\mathinner{\rm div}}

\newcommand{\spt}{\mathop{\rm Supp}}

\newcommand{\pt}{\partial}

\newcommand{\Lip}{\mathinner\mathbf{Lip}}

\renewcommand{\d}[1]{\mathinner{\mathrm{d}{#1}}}

\newcommand{\pd}{{\eta}}
\newcommand{\gd}{{\lambda}}

\makeatletter

\@addtoreset{equation}{section}
\makeatother

\title{Improved stability estimates on general scalar balance laws}

\author{Magali L\'ecureux-Mercier$^\textrm{a}$}

\begin{document} 
\maketitle

\footnotetext[1]{Laboratoire MAPMO, 
Universit\'e d'Orl\'eans, UFR Sciences, 
B\^atiment de math\'ematiques - Rue de Chartres, 
B.P. 6759 - 45067 Orl\'eans cedex 2
France}
\begin{abstract}
  \noindent Consider the general scalar balance law $\partial_t u +
  \Div f(t, x,u) = F(t,x,u)$ in several space dimensions. The aim of
  this note is to improve the results of  Colombo,  Mercier,  Rosini who gave an estimate of the 
dependence of the solutions  from the  flow $f$ and from the source $F$. The improvements are twofold: 
first the  expression of the coefficients in these estimates  are more precise; second,  we eliminate some 
regularity hypotheses thus extending significantly the applicability of our estimates.

  \smallskip

  \noindent\textit{2000~Mathematics Subject Classification:} 35L65.

  \smallskip


\end{abstract}

\section{Introduction}

We consider the Cauchy problem for the general scalar balance law
\begin{equation}
  \label{eq:probv}
  \left\{\begin{array}{l@{\qquad}rcl}
      \partial_t u + \Div  f(t,x,u) = F(t,x,u)
      & (t,x) & \in & \rpis\times\reali^N
      \\
      u(0,x)=u_0(x) & x & \in & \reali^N\,.
      \\
    \end{array}
  \right.
\end{equation}
This kind of equation has already been intensively studied: a fundamental result is the one of S. N.  Kru\v zkov \cite[Theorem~1 \& 5]{Kruzkov}, stating the existence and uniqueness of a weak entropy solution for an initial data $u_0\in \L\infty(\reali^N,\reali)$. In addition, Kru\v zkov describes the dependence of the solutions with respect to the initial condition: if $u_0$ and $v_0$ are two initial data, then the associated entropy solutions $u$ and $v$ satisfy
\begin{equation}\label{eq:kru}
 \norma{(u-v)(t)}_{\L1}\leq e^{\gamma t}\norma{u_0-v_0}_{\L1}\,, \qquad\qquad \textrm{with } \gamma=\norma{\pt_u F}_{\L\infty}\,.
\end{equation}
A huge literature on this subject is available in the special case the flow $f$ depends only on $u$ and not on the variables $t$ and $x$  and there is no source $F=0$  (see for example \cite{Bressan, DafermosBook, Sr1, Sr2}).

We are interested here in the dependence of the solution with respect to flow $f$ and source $F$ in the case these functions depend on the three variables $t$, $x$ and $u$.

This dependence with respect to flow and source has  already been investigated: 
this question was first addressed from the point of view of numerical analysis by B. Lucier \cite{lucier} who studied the case of an homogeneous flow ($f(u)$), without source term ($F=0$). 
More recently F. Bouchut \& B. Perthame \cite{bouchutperthame} improved this result, always in the case of an homogeneous flow  and without source. 
G.-Q. Chen \& K.  Karlsen \cite{chenkarlsen} also studied this dependence, for a flow depending also on $x$, 
but the estimate they obtained was depending on an a priori (unknown) bound on $\tv(u(t))$. 

The purpose of the present paper is to improve the recent result of R. Colombo, M. Mercier \& M. Rosini \cite{ColomboMercierRosini}, which provided an estimate  of the total variation in the general case (with flow and source  depending on the three variables $t$, $x$ and $u$) and of the $\L1$ distance  between solutions. 
In particular, this estimate can be compared to the one of Kru\v zkov (\ref{eq:kru}) which gives a bound  on the $\L1$ distance between solutions with different initial data (but with same flow and source). The estimates (\ref{eq:kru}) and \cite[Theorem 2.6]{ColomboMercierRosini} look similar but in \cite{ColomboMercierRosini}, the coefficient $\gamma$ given by Kru\v zkov in (\ref{eq:kru}) is replaced by $\kappa=2N\norma{\nabla\pt_u f}_{\L\infty}+\norma{\pt_u F}_{\L\infty}$. Consequently,   we do not recover (\ref{eq:kru}) from \cite{ColomboMercierRosini} in the case $F=0$ (because $\gamma=0$ whereas $\kappa= 2N\norma{\nabla\pt_u f}_{\L\infty}\neq 0$  a priori).

In the same setting as in~\cite{ColomboMercierRosini, Kruzkov}, we provide here an
estimate on the total variation of the solution to  (\ref{eq:probv}), and on the dependence of the solutions to~(\ref{eq:probv}) on
the flow $f$ and on the source $F$, with better hypotheses and coefficients than in \cite{ColomboMercierRosini}. 
The advances are twofold. Firstly, we relax hypotheses, and thus widely extend the usability of our results. %Moreover, we obtain here a precise estimate of the $\L1$ distance between two solutions (for possibly different initial data) of a same equation, in terms of the flow $f$ and on the source $F$. 
More precisely, we require here less regularity in time than in \cite{ColomboMercierRosini}, which is very useful for applications (see \cite{chm_control, ColomboMercier}). Furthermore, we recover the same estimate as Kru\v zkov when we consider the dependence toward initial conditions only.

This note is organized as follows. In Section \ref{sec:thm} we state the main results and compare them to those in \cite{ColomboMercierRosini}. In Section \ref{sec:prelim}, we give some tools on functions with bounded variations; in Sections \ref{sec:prooftv} and \ref{sec:proofcomparison} we  prove Theorems \ref{teo:tv} and \ref{teo:estimates}; finally Section \ref{sec:tech}
contains some technical lemmas used in the preceding sections.

\section{Main results}\label{sec:thm}

We shall use the notations $\rpic = \left[0, + \infty \right)$ and $\rpis = \left(0,
  +\infty \right)$. Below, $N$ is a positive integer, $\Omega = \rpis
\times \reali^N \times \reali$; for any positive $T$, $U$ we denote   $\Omega_T^U=[0,T]\times\reali^N\times [-U,U]$;  $B(x,r)$ stands for  the ball in
$\reali^N$ with center $x \in \reali^N$ and radius $r > 0$ and $\spt(u)$ stands for  the support of $u$. The volume
of the unit ball $B(0,1)$ is $\omega_N$. For notational simplicity, we
set $\omega_0 = 1$.  The following induction formula gives $\omega_N$ in terms of the Wallis integral $W_N$:
\begin{equation}
  \label{eq:W}
  \frac{\omega_{N}}{\omega_{N-1}}
  =
  2 \, W_N
  \qquad \mbox{ where } \qquad
  W _N
  =
  \int_0^{\pi/2} (\cos \theta)^N \, \mathrm{d}\theta \,.
\end{equation}
In the present work, $\mathbf{1}_A$ is the characteristic function of
the set $A$,  and $\delta_t$ is the Dirac measure centered at
$t$. Besides, for a vector valued function $f = f(x,u)$ with $u =
u(x)$, $\Div f$ stands for the total divergence. On the other hand,
$\div f$, respectively $\nabla f$, denotes the partial divergence,
respectively gradient, with respect to the space variables. Moreover,
$\partial_u$ and $\partial_t$ are the usual partial derivatives. Thus,
$\Div f = \div f + \partial_u f \cdot \nabla u$.

The following sets of assumptions on $f$ and $F$ will be of use below.
\begin{eqnarray}
    \mathbf{(H1^*)}\!\!
    &&
    \left\{
      \begin{array}{l}
        f \in \C0(\Omega; \reali^N)\,,\qquad
        F \in \C0(\Omega; \reali)\,,
        \\[5pt]
        f, F \textrm{ have  continuous derivatives }\pt_u f\,,\; \pt_u\nabla f\,,\; \nabla^2 f\,,\; \pt_u F\,,\; \nabla F\,;\\[5pt]
	\textrm{for all }U, \,T>0, \qquad\pt_u f\in \L\infty(\Omega_T^U  ; \reali^N)\,,\\[5pt]
        F-\div f \in \L\infty(\Omega_T^U  ;\reali)\,,\qquad  \pt_u(F-\div f) \in \L\infty(\Omega_T^U  ;\reali)\,.
      \end{array}
    \right.
    \\
    \mathbf{(H2^*)}\!\!
    &&
    \left\{
      \begin{array}{l}
\textrm{for all }U,\; T>0\,,\qquad\nabla\pt_u f\in \L\infty(\Omega_T^U  ;\reali^{N\times N})\,,\qquad \pt_u F\in \L\infty(\Omega_T^U  ;\reali)\,,\\[5pt]
% \textrm{for all } t>0\,,\; u\in \reali\,,\quad\nabla^2 f(t,\cdot,u) \in \L1(\reali^N;\reali^{N\times N\times N})\,,\\[5pt]
%\textrm{for all }U,\, T>0\,,\qquad
\displaystyle \int_0^T\!\!\int_{\reali^N}\norma{\nabla(F-\div f)(t,x,\cdot)}_{\L\infty([-U, U];\reali^N)} \d{x}\,\d{t}< \infty\,.
      \end{array}
    \right.
    \\[2pt]
    \mathbf{(H3^*)}\!\!
    &&
    \left\{
      \begin{array}{l}
      \textrm{for all }U,\, T>0 \qquad \pt_u f\in \L\infty(\Omega_T^U  ;\reali^N)\,,\qquad \pt_u F\in \L\infty(\Omega_T^U  ;\reali)\,,\\[5pt]
        \displaystyle
%       \textrm{for all }U,\,T>0\,,\quad 
\int_{0}^T \!\! \int_{\reali^N} \!
        \norma{(F- \div f)(t, x, \cdot)}_{\L\infty([-U,U];\reali)} 
        \, \mathrm{d}x\, \mathrm{d}t < + \infty\,.
      \end{array}
    \right.
\end{eqnarray}

Comparing these sets of hypotheses to \textbf{(H1)},\textbf{(H2)} and \textbf{(H3)} in \cite{ColomboMercierRosini}, we note that
\begin{itemize}
\item no derivatives in time are now needed;
\item the $\L\infty$ norm are now taken on the domain $\Omega^U_T=[0,T]\times \reali^N\times [-U,U]$ which is smaller than  $\Omega=\rpic\times\reali^N\times \reali$, which was the domain considered in \cite{ColomboMercierRosini}.
\end{itemize}

Let us recall the fundamental theorem 
\begin{theorem}[Kru\v zkov \cite{Kruzkov}]
  \label{teo:Kruzkov}
  Assume~$\mathbf{(H1^*)}$ hold. Then, for any $u_0 \in \L\infty (\reali^N;
  \reali)$, there exists a unique weak entropy solution $u$
  to~(\ref{eq:probv}) in $\L\infty \left( \rpic; \Lloc{1}(\reali^N;
    \reali) \right)$ continuous from the right. Moreover, if a
  sequence $u_0^n \in \L\infty (\reali^N; \reali)$ converges to $u_0$
  in $\Lloc1$, then for all $t>0$ the corresponding solutions $u^n(t)$
  converge to $u(t)$ in $\Lloc1$.
\end{theorem}

\subsection{Estimate on the Total Variation}

We give here a result similar to the one obtained by Colombo, Mercier and Rosini \cite[Theorem 2.5]{ColomboMercierRosini}, but under weaker assumptions.

\begin{theorem}
  \label{teo:tv}
  Assume that~$\mathbf{(H1^*)}$ and~$\mathbf{(H2^*)}$ hold. Let $u_0 \in
  (\L\infty\cap\L1\cap\BV)(\reali^N; \reali)$. Then, the weak entropy solution
  $u$ of~(\ref{eq:probv}) satisfies $u(t) \in \BV(\reali^N; \reali)$
  for all $t > 0$. Let $T_0$ be real positive. Let us denote $\mathcal{U}=\norma{u}_{\L\infty([0,T_0]\times \reali^N)}$, 
${U}_t= \sup_{y\in \reali^N}\modulo{u(t,y)}$, $\mathcal{S}_{T_0}(u)=\bigcup_{t\in [0,T_0]} \spt(u(t))$  and  
  \begin{eqnarray}
\Sigma_{T_0}^u&=&[0,T_0]\times \mathcal{S}_{T_0}(u)\times[-\mathcal{U}, \mathcal{U}]\,,\label{eq:KTU}\\
    \label{eq:kappao}
    \kappa^*_0
     &=&
      (2N+1) \, \norma{\nabla \, \partial_u
        f}_{\L\infty(\Sigma_{T_0}^u ;\reali^{N\times N})}
      +
      \norma{\partial_u F}_{\L\infty(\Sigma_{T_0}^u ;\reali)}
  \end{eqnarray}
  then for all $T\in  [0,T_0]$,  with $W_N$ as in~(\ref{eq:W}),
  \begin{equation}
    \label{result}
    \tv \left( u(T) \right)
    \leq
    \displaystyle
    \tv(u_0) \, e^{\kappa^*_0 T}
    \displaystyle
    +
    N W_N \!\! \int_0^T \!\!e^{\kappa^*_0(T-t)}\!\! \int_{\reali^N}\!
    \norma{\nabla( F - \div f)(t, x, \cdot)}_{\L\infty([-{U}_t,{U}_t];\reali)}
    \mathrm{d}x\, \mathrm{d}t \,.
  \end{equation}
\end{theorem}

\begin{remark}
Note that, with $c=\norma{\pt_u f}_{\L\infty(\Omega_{T_0}^{\mathcal{U}})}$, we have $\spt u(t)\subset \spt u_0 +B(0,ct)$. Consequently,
$$
\mathcal{S}_{T_0}(u)\subset \spt u_0+ B(0,c\,T_0)\,.
$$ 
\end{remark}

We can note here several improvements with respect to \cite[Theorem 2.5]{ColomboMercierRosini}. 
First, as we already noted, the set of hypotheses is weaker since we do not require $f$ to be $\C2$ and $F$ to be $\C1$ with respect 
to the time variable: 
they only have to be continuous in time, which is useful in applications, see for example  \cite{chm_control}.

A second improvement stands in  the $\L\infty$ norms, that are taken on  smaller domains than in \cite{ColomboMercierRosini}. 

Last, the expression of the coefficient $\kappa^*_0$ that does not content any 
longer the constant $N W_N$. Indeed, in \cite[Theorem 2.5]{ColomboMercierRosini} it was given by 
\[
\kappa_0=NW_N\left((2N+1) \norma{\nabla\pt_u f}_{\L\infty(\Omega;\reali^{N\times N})}+ \norma{\pt_u F}_{\L\infty(\Omega;\reali)}\right)
\]

Besides, it does not  seem  possible to erase the coefficient $NW_N$ completely from the expression  
(\ref{result}), except in the case  $F$ and $f$ do not depend on $u$,  see Remark \ref{rem:opt}.

An important corollary of this theorem is that we have now a criterium for having solution continuous in time instead of continuous from the right. This is the analogous of \cite[Theorem 4.3.1]{DafermosBook} for general flows and sources. We use here the  same notations as in Theorem \ref{teo:tv}.
\begin{corollary}\label{cor:lip}
 Assume that $(f,F)$ satisfy~$\mathbf{(H1^*)}$, $\mathbf{(H2^*)}$ and $\mathbf{(H3^*)}$. Let $u_0 \in
  (\L\infty\cap\L1\cap\BV)(\reali^N; \reali)$ and let $u$ be the weak entropy solution of (\ref{eq:probv}). Then
$
 u\in \C0([0,T], \L1(\reali^N;\reali))
$ and for any $s,t\in [0,T]$ we have the estimate
\begin{eqnarray}
 \norma{u(t)-u(s)}_{\L1}&\leq& \modulo{ \int_s^t\int_{\reali^N}
    \norma{( F - \div f)(\tau, x, \cdot)}_{\L\infty([-\mathcal{U},\mathcal{U}];\reali)}
    \mathrm{d}x } \label{eq:lip}\\[5pt]
&&+ \modulo{s-t}\norma{\pt_u f}_{\L\infty(\Sigma_{T}^u)}\sup_{\tau\in [0,T]} \tv(u(\tau))\,.\nonumber
\end{eqnarray}

If furthermore, for $T_0>0$, instead of \textbf{(H3*)}, the condition 
$$\sup_{t\in [0,T_0]}\int_{\reali^N}
    \norma{( F - \div f)(t, x, \cdot)}_{\L\infty([-\mathcal{U},\mathcal{U}];\reali)}
    \mathrm{d}x\, \mathrm{d}t<\infty$$ 
holds, then the application $t\in [0,T_0]\to u(t,\cdot)\in \L1(\reali^N,\reali)$ is Lipschitz.
\end{corollary}

\subsection{Stability of Solutions with Respect to Flow and Source}

We want now to estimate the difference $u-v$, where 
\begin{itemize}
 \item $u$ is the solution of (\ref{eq:probv}) with flow $f$, source $F$ and initial condition $u_0$,
 \item  $v$ is the solution of (\ref{eq:probv}) with flow $g$, source $G$ and initial condition $v_0$.
\end{itemize}
We search for an estimate of $u-v$ in term of $f-g$, $F-G$ and $u_0-v_0$.

F. Bouchut \& B. Perthame in~\cite{bouchutperthame} obtained such an estimate in the particular case
 $f$, $g$ depend only on $u$ and $F = G = 0$. The following result is an improvement of the result of R. Colombo, M. Mercier and M. Rosini \cite[Theorem 2.6]{ColomboMercierRosini}, in which we  gave a similar result under stronger assumptions and with a coefficient $\kappa^*$ that was not compatible with the result of Kru\v zkov (\ref{eq:kru}).

\begin{theorem}
  \label{teo:estimates}
  Let $(f, F)$, $(g,G)$ satisfy~$\mathbf{(H1^*)}$, $(f,F)$ satisfy
  $\mathbf{(H2^*)}$ and $(f-g,F-G)$ satisfy~$\mathbf{(H3^*)}$. Let $u_0, v_0
  \in \L\infty\cap\L1\cap\BV(\reali^N; \reali)$.  Let $T>0$ and let us denote 
\begin{equation}\label{eq:def}
\begin{array}{rl}
 \mathcal{V}&=\max(\norma{u}_{\L\infty([0,T]\times \reali^N)},\norma{v}_{\L\infty([0,T]\times\reali^N)})\,,\\[5pt]
V_t&=\sup_{y\in \reali^N}(\modulo{u(t,y)}, \modulo{v(t,y)})\,, \\[5pt]
\mathcal{S}_T(u,v)&=\bigcup_{t\in [0,T]} \left( \spt u(t) \cup \spt v(t)\right)\,,\\[5pt]
 \Sigma_{T}^{u,v}&=[0,T]\times\mathcal{S}_T(u,v)\times [-\mathcal{V},\mathcal{V}]\,.
\end{array}
\end{equation}

Furthermore, we define  $\kappa^*_0$, $U_t$,   $\Sigma_T^u$ as  in~(\ref{eq:kappao}) and  
  \begin{align} 
  \kappa^*
 &   =
    \norma{\pt_u F}_{\L{\infty}(\Sigma_T^{u,v} ; \reali)}+\norma{\pt_u \div(g-f)}_{\L\infty(\Sigma_{T_0}^{u,v})}\,, &
    M
&    =
    \norma{\partial_u g}_{\L\infty(\Omega_T^\mathcal{V};\reali^N)}\,. \label{eq:KM}
  \end{align}
  Then,  for any $R >0$ and $x_0 \in \reali^N$, the following
  estimate holds:
  \begin{eqnarray*}
    & &
    \int_{\norma{x-x_0}\le R} \modulo{u(T,x)-v(T,x)}\mathrm{d}x
    \; \leq \;
    e^{\kappa^* T}
    \int_{\norma{x-x_0}\leq R + M T} \modulo{u_0(x) - v_0(x)} 
    \, \mathrm{d}x
    \\
    & + &
    \frac{e^{\kappa^*_0 T}-e^{\kappa^* T}}{\kappa^*_0-\kappa^*}  \, \tv(u_0) \, 
    \norma{\pt_u(f-g)}_{\L{\infty}(\Sigma_T^u;\reali^N)}  
    \\
    & + &
    NW_N 
    \left(
      \int_0^T \frac{e^{\kappa^*_0 (T-t)}-e^{\kappa^* (T-t)}}{\kappa^*_0-\kappa^*} 
      \int_{\reali^N} \norma{\nabla(F-\div f)(t, x,\cdot)}_{\L\infty([-{U}_t,{U}_t])}
      \mathrm{d}x \, \mathrm{d}t
    \right)\\
&&\qquad \times
    \norma{\pt_u(f-g)}_{\L{\infty}(\Sigma_T^u;\reali^N)}
    \\
    & + &
    \int_0^T e^{\kappa^* (T-t)} \int_{\norma{x-x_0}\leq R+M(T-t)}
    \norma{\left((F-G) - \div(f-g) \right)(t,x,\cdot)}_{\L\infty([-{V}_t,{V}_t])} 
    \, \mathrm{d}x\, \mathrm{d}t \,.
  \end{eqnarray*}
\end{theorem}

This theorem is a direct consequence of lemma \ref{lem:main}.

\begin{remark}
Note as above that, with $c'=\max(\norma{\pt_u f}_{\L\infty(\Omega_{T_0}^{\mathcal{U}})}, \norma{\pt_u g}_{\L\infty(\Omega_{T_0}^{\mathcal{V}})})$, we have $\spt u(t)\subset \spt u_0 +B(0,c' \,t)$ and $\spt v(t)\subset \spt v_0 +B(0,c'\, t)$. Consequently,
$$
\mathcal{S}_{T}(u, v)\subset (\spt u_0\cup\spt v_0)+ B(0,c'\,T)\,.
$$ 
\end{remark}

\begin{remark}
As above, we can note some improvements with respect to \cite[Theorem 2.6]{ColomboMercierRosini}:
\begin{itemize}
 \item The hypotheses are weaker: no derivative in time is needed for $f$ and $F$.
\item The $\L\infty$ norms are taken on  smaller domains.
\item The coefficient $\kappa^*$ is better than the $\kappa$ given in \cite[Theorem 2.6]{ColomboMercierRosini} by
\[
\kappa= 2N\norma{\nabla\pt_u f}_{\L\infty(\Omega;\reali^{N\times N})} +\norma{\pt_u F}_{\L\infty{(\Omega;\reali)}}+ \norma{\pt_u(F-G)}_{\L\infty(\Omega;\reali)}\,.
\]
Indeed, $\kappa^*$ coincides with $\gamma$ in the case $f=g$ and consequently we recover the previous Kru\v zkov's result (\ref{eq:kru}), which was not the case with $\kappa$. 
\end{itemize}
\end{remark}

\begin{remark}
Note that, if $\kappa_0^*\geq \kappa^*$ then
\[
\frac{e^{\kappa^*_0 t}-e^{\kappa^* t}}{\kappa^*_0-\kappa^*}\leq e^{\kappa^* t}\int_0^te^{(\kappa_0^*-\kappa^*)\tau}\d{\tau} \leq te^{\kappa_0^* t}\,.
\]
As the expression is symmetric, we can conclude in the general case that, denoting $\kappa_1=\max (\kappa_0^*, \kappa^*)$, we have 
$
\frac{e^{\kappa^*_0 t}-e^{\kappa^* t}}{\kappa^*_0-\kappa^*}\leq te^{\kappa_1 t}
$. Let us assume that $\kappa_0^*\geq \kappa^*$; then the estimate of Theorem \ref{teo:estimates} can be rewritten
  \begin{eqnarray*}
    & &
    \int_{\norma{x-x_0}\le R} \modulo{u(T,x)-v(T,x)}\mathrm{d}x
    \; \leq \;
    e^{\kappa^* T}
    \int_{\norma{x-x_0}\leq R + M T} \modulo{u_0(x) - v_0(x)} 
    \, \mathrm{d}x
    \\
    & + &
    T  \tv(u(T)) \, 
    \norma{\pt_u(f-g)}_{\L{\infty}(\Sigma_T^u;\reali^N)}  
    \\
    & + &
   e^{\kappa^* T} \int_0^T  \int_{\norma{x-x_0}\leq R+M(T-t)}
    \norma{\left((F-G) - \div(f-g) \right)(t,x,\cdot)}_{\L\infty([-{V}_t,{V}_t])} 
    \, \mathrm{d}x\, \mathrm{d}t \,.
  \end{eqnarray*}
\end{remark}

Another consequence of Lemma \ref{lem:main} is the following proposition.

\begin{proposition}\label{thm:stab2}
Let $(f, F)$, $(g,G)$ satisfy~$\mathbf{(H1^*)}$, $(f,F)$ satisfy
  $\mathbf{(H2^*)}$ and $(f-g,F-G)$ satisfy~$\mathbf{(H3^*)}$. Let $u_0, v_0
  \in \L\infty\cap\L1\cap\BV(\reali^N; \reali)$. Let $T>0$.
  Then, using the same notation as in (\ref{eq:def})--( \ref{eq:KM}),  for any $R >0$ and $x_0 \in \reali^N$, the following
  estimate holds:
  \begin{eqnarray*}
    & &
    \int_{\norma{x-x_0}\le R} \modulo{u(T,x)-v(T,x)}\mathrm{d}x
    \; \leq \;
    e^{\kappa^* T}
    \int_{\norma{x-x_0}\leq R + M T} \modulo{u_0(x) - v_0(x)} 
    \, \mathrm{d}x
    \\
    & + &
  \left[\tv(u_0)
  +NW_N  \int_0^T e^{-\kappa_0^* t}\int_{\reali^N} \norma{\nabla(F-\div f)(t,x,\cdot)}_{\L\infty([-U_t, U_t])}\d{x}\d{t}\right] \\
  &&\qquad\times  \frac{\kappa_0^*e^{\kappa_0^* t}- \kappa^*e^{\kappa^* t}}{\kappa_0^*-\kappa^*}\int_0^T \norma{\pt_u(f-g)(t)}_{\L\infty(\mathcal{S}_{T}\times[-V_t, V_t])}\d{t}\\
  &+&e^{\kappa^* T} \int_0^T \int_{\norma{x-x_0}\leq R+M(T-t)}
    \norma{\left((F-G) - \div(f-g) \right)(t,x,\cdot)}_{\L\infty([-{V}_t,{V}_t])} 
    \, \mathrm{d}x\, \mathrm{d}t \,.
  \end{eqnarray*}
\end{proposition}

This proposition  is useful in \cite{ColomboGaravelloLecureux}, where we studied the equation
\[
 \pt_t u +\div(u(1-u) w(u *_x \eta))=0\,,\qquad \qquad u(0,\cdot)=\bar u\,,
\]
and in particular, the stability with respect to $\eta$. The use of proposition \ref{thm:stab2} allows 
then to apply Gronwall lemma and gives us the following stability result. We assume here that we have existence and uniqueness of weak entropy solutions, as obtained in  \cite{ColomboGaravelloLecureux}.
\begin{proposition}
Let $w\in\Lip(\reali, \reali)$ be such that $w'\in \W1\infty(\reali, \reali)$, $\eta_1, \eta_2\in \W21\cap\W1\infty(\reali^N, \reali)$, $\bar u\in \L1\cap\L\infty\cap\BV(\reali^N, [0,1])$. Let $u_1, u_2\in \C0(\rpic, \L1(\reali^N,  [0,1] )) $ be weak entropy solutions to the Cauchy problems (for $i=1, 2$):
\[
  \pt_t u_i +\div(u_i(1-u_i) w(u_i *_x \eta_i))=0\,,\qquad \qquad u_i(0,\cdot)=\bar u\,.
\]
 Then, we have the  stability estimate:
 \[
 \norma{(u_1-u_2)(t)}_{\L1}\leq C(t)\norma{\eta_1-\eta_2}_{\W11}\,,
\]
where $C$ depends on $\norma{\bar u}_{\L1}, \norma{u_1}_{\L\infty([0,T]\times \reali^N)}, \norma{u_2}_{\L\infty([0,T]\times \reali^N)}$  and on various norms on $\eta$ and $w$.
\end{proposition}

\begin{proof}
 Applying Theorem \ref{thm:stab2}, we obtain
\begin{align*}
\norma{u_1(T)-u_2(T)}_{\L1}
\leq & a(T)+b(T)\int_0^T \norma{u_1-u_2(t)}_{\L1}\d{t}
\end{align*}
where $a$ and $b$ are regular and increasing functions of $T$.
Applying Gronwall Lemma, we obtain the desired estimate.
\end{proof}

\section{Tools on functions with  bounded variation}
\label{sec:prelim}

Recall  the following theorem (see~\cite[Theorem~3.9 and
Remark~3.10]{AmbrosioFuscoPallara}):

\begin{theorem}
  \label{thm:AFP}
  Let $u \in \Lloc1(\reali^N;\reali)$. Then $u \in \BV(
  \reali^N;\reali)$ if and only if there exists a sequence $(u_n)$ in
  $\C\infty (\reali^N;\reali)$ converging to $u$ in $\Lloc1$ and
  satisfying
  \begin{displaymath}
    \lim_{n\to +\infty} \int_{\reali^N} \norma{\nabla u_n (x)}\, \mathrm{d}x
    =
    L \quad \mbox{ with } \quad
    L < \infty \,.
  \end{displaymath}
  Moreover, $\tv(u)$ is the smallest constant $L$ for which there exists
  a sequence as above.
\end{theorem}

Let us also recall the following property of any function $u \in
\BV(\reali^N;\reali)$:
\begin{equation}
  \label{eq:afp}
  \int_{\reali^N} \modulo{u(x)-u(x-z)} \, \mathrm{d}x
  \leq
  \norma{z} \, \tv(u) 
  \qquad \mbox{ for all } z \in \reali^N .
\end{equation}
For a proof, see~\cite[Remark~3.25]{AmbrosioFuscoPallara}.

Now, in a similar way as  J. D\'avila \cite{davila}, we prove the following proposition, which is an improvement of \cite[Proposition 4.3]{ColomboMercierRosini}. Indeed, in \cite[Proposition 4.3]{ColomboMercierRosini},  the equality (\ref{eq:limm}) is valid only for $u\in \C1$. In the present proposition we extend this result to all $u\in \BV$.

\begin{proposition}
  \label{prop:normtv2}
Let $\rho_1\in \Cc\infty(\reali,\rpic)$ with $\spt \rho_1\subset [-1,1]$. Let  $u \in \Lloc1(\reali^N;  \reali)$. 
For all  $\gd >0$, we introduce $\rho_\gd$ such that $\rho_\lambda(x)=\frac{1}{\gd^N}\rho_1\left(\frac{\norma{x}}{\gd}\right)$. 
Assume that there exists a constant $\widetilde C$  such that for all
  $\gd$, $R$ positive,
  \begin{equation}
    \frac{1}{\gd} \int_{\reali^N} \int_{B(x_0,R)}
    \modulo{u(x)-u(x-z)} \, \rho_\gd(z)
    \, \mathrm{d}x \, \mathrm{d}z
    \leq
    \widetilde C\,.
  \end{equation}
 Then $u \in \BV(\reali^N; \reali)$ and
  \begin{equation}\label{eq:limm}
    \tv (u)
    =
    \frac{1}{C_1}
    \lim_{\gd \to 0}
    \frac{1}{\gd} \int_{\reali^N} \int_{\reali^N}
    \modulo{u(x)-u(x-z)} \, \rho_\gd(z)
    \, \mathrm{d}x \, \mathrm{d}z \,,
  \end{equation} 
where
\begin{equation}
    \label{eq:Cst0}
    C_1
    =
    \int_{\reali^N} \modulo{x_1} \, \rho_1\left(\norma{x} \right) 
    \, \mathrm{d}x \,.
  \end{equation}
\end{proposition}

\begin{proof}

Note that the first part of the proof is the same as the first part of the proof of \cite[Proposition 4.3]{ColomboMercierRosini}. 
We introduce a regularisation of  $u$:   $u_h = u \ast \mu_h$, with
  $\mu_h(x) = \mu_1 \left( \norma{x}/h \right) / h^N$, where $\mu_1$  is defined as in  \Ref{eq:mu}. We note that $u_h
  \in \C\infty (\reali^N; \reali)$ and that  $u_h$  tends to   $u$ in 
  $\Lloc1$ when $h \to 0$.
Furthermore,  for  $R$ and $h$ positive,  by change of variables we get
  \begin{eqnarray*}
&& \int_{\reali^N} \int_{B(x_0,R-h)}
    \modulo{ \int_0^1 \nabla u_h(x-\gd s z) \cdot z \,\mathrm{d}s } \, \rho_1(\norma{z})
    \, \mathrm{d}x \, \mathrm{d}z \\
&=&    \frac{1}{\gd}
    \int_{\reali^N} \int_{B(x_0,R-h)}
    \modulo{u_h(x)-u_h(x-\gd z)} \, \rho_1(\norma{z})
    \, \mathrm{d}x \, \mathrm{d}z
    \\
    & \leq &
    \frac{1}{\gd}
    \int_{\reali^N} \int_{B(x_0,R)}
    \modulo{u(x)-u(x-\gd z)} \, \rho_1(\norma{z})
    \, \mathrm{d}x \, \mathrm{d}z
    \\
    & \leq &
    \widetilde C\,.
  \end{eqnarray*}
Making $R\to \infty$ and using the Dominated Convergence Theorem when $\gd \to 0$, we obtain
\begin{eqnarray*}
&& \int_{\reali^N} \int_{\reali^N}
    \modulo{ \nabla u_h(x) \cdot z } \, \rho_1(\norma{z})
    \, \mathrm{d}x \, \mathrm{d}z\\
&\leq& \liminf_{\gd\to 0} \frac{1}{\gd}
    \int_{\reali^N} \int_{B(x_0,R)}
    \modulo{u(x)-u(x-\gd z)} \, \rho_1(\norma{z})
    \, \mathrm{d}x \, \mathrm{d}z\,.
\end{eqnarray*}

 Remark that for fixed $x\in \reali^N$, when $\nabla u_h(x) \neq 0$,
  the scalar product $\nabla u_h(x)\cdot z$ is positive (respectively,
  negative) when $z$ is in a half-space, say $H_x^+$ (respectively,
  $H_x^-$). We can write $z = \alpha \frac{\nabla
    u_h(x)}{\norma{\nabla u_h(x)}} + w$, with $\alpha \in \reali$ and
  $w$ in the hyperplane $H^o_x = \nabla u_h(x)^\perp$. Hence
  \begin{eqnarray*}
    \int_{\reali^N}
    \modulo{\nabla u_h(x) \cdot z} \, \mu_1(\norma{z}) \, \mathrm{d}z
    & = &
    \int_{H^+_x} \!\!
    \nabla u_h(x) \cdot z \, \mu_1(\norma{z}) \,
    \mathrm{d}z
    +
    \int_{H^-_x} \!\!
    \nabla u_h(x) \cdot (-z) \, \mu_1(\norma{z}) \,
    \mathrm{d}z
    \\
    & = &
    2 \int_{H^+_x}
    \nabla u_h(x) \cdot z \, \mu_1(\norma{z}) \, \mathrm{d}z
    \\
    & = &
    2 \int_{\rpis} \int_{H^o_x}
    \alpha \, \norma{\nabla u_h(x)} \,
    \mu_1(\sqrt{\alpha^2+\norma{w}^2})
    \, \mathrm{d}w \, \mathrm{d}\alpha
    \\
    & = &   \int_{\reali} \int_{H^o_x}
    \modulo{\alpha} \, \norma{\nabla u_h(x)} \,
    \mu_1(\sqrt{\alpha^2+\norma{w}^2})
    \, \mathrm{d}w \, \mathrm{d}\alpha
    \\
    & = & \norma{\nabla u_h(x)} \int_{\reali^N} \modulo{z_1} \,
    \mu_1(\norma{z}) \, \mathrm{d}z\, .
  \end{eqnarray*}
 So we obtain
\begin{equation}\label{eq:liminf}
 \tv(u) \leq \frac{1}{ C_1} \liminf_{\gd\to 0}  \frac{1}{\gd}
    \int_{\reali^N} \int_{\reali^N}
    \modulo{u(x)-u(x-z)} \, \rho_\gd(z)
    \, \mathrm{d}x \, \mathrm{d}z \leq \frac{ \widetilde C}{ C_1}\,.
\end{equation}

Now, let  $(u_n)$ be a sequence of functions in  $\C\infty(\reali^N,\reali)$ converging to $u$ in  $\Lloc1$ and such that   $\int_{\reali^N} \norma{\nabla u_n(x)}\d{x}$ converges to $\tv(u)$ when $n\to \infty$.
Then, doing the same computation as above, we obtain
\begin{eqnarray*}
 & \displaystyle\frac{1}{\gd}
    \int_{\reali^N} \int_{B(x_0,R)}
    \modulo{u_n(x)-u_n(x-\gd z)} \, \rho_1(\norma{z})
    \, \mathrm{d}x \, \mathrm{d}z
    \\
   \leq&  \displaystyle  \int_{\reali^N}\int_{B(x_0,R)}\int_0^1\modulo{\nabla u_n (x-\gd s z)\cdot z}\rho_1(\norma{z})\d{s}\d{x}\d{z}\\
\leq& \displaystyle\int_{\reali^N}\int_0^1 \int_{B(x_0,R+\gd)} \modulo{\nabla u_n (x')\cdot z}\rho_1(\norma{z})\d{x'}\d{s}\d{z}\\
=& \displaystyle \int_{B(x_0,R+\gd)}\norma{\nabla u_n (x)} C_1\d{x}\\
\leq& C_1\tv(u_n, B(x_0,R+\lambda)) \,.
\end{eqnarray*}
Taking $R\to\infty$ and then $n\to \infty$,  we have consequently
\[
\frac{1}{\gd}
    \int_{\reali^N} \int_{\reali^N}
    \modulo{u(x)-u(x-\gd z)} \, \rho_1(\norma{z})
    \, \mathrm{d}x \, \mathrm{d}z \leq C_1 \tv(u)\,.
\]
Then, we % make $R$ goes to $+\infty$  and we 
take the supremum limit when $\gd$ goes to 0. We obtain 
\begin{equation}\label{eq:limsup}
\limsup_{\gd\to 0} \frac{1}{\lambda} \int_{\reali^N}\int_{\reali^N}\modulo{u(x)-u(x-z)}\rho_\gd(z)\d{x}\d{z}\leq C_1\tv(u)\,.
\end{equation}

We conclude the proof by reassembling  (\ref{eq:liminf}) and (\ref{eq:limsup}).
\end{proof}

\section{Proof of the Total Variation estimate}\label{sec:prooftv}

The following proof is quite similar to the one of \cite[Theorem 2.5]{ColomboMercierRosini}. 
The differences come from the use of Proposition \ref{prop:normtv2} instead of \cite[Proposition 4.3]{ColomboMercierRosini} 
and from avoiding the derivatives in time to appear. In order to be clear, we rewrite here most of the steps of the proof. In particular, the beginning of the proof is similar to \cite[proof of Theorem 2.5]{ColomboMercierRosini} up to (\ref{eq:nu}).

\begin{proofof}{Theorem~\ref{teo:tv}}
  First, we assume that $u_0 \in \C1(\reali^N;\reali)$. The general case
  will be considered only at the end of this proof.

By Kru\v zkov Theorem \cite[Theorem 5 \& Section 5 Remark 4]{Kruzkov}, the set of hypotheses \textbf{(H1*)} gives us existence and uniqueness of a weak entropy solution for any initial condition $u_0\in \L\infty\cap\L1(\reali^N;\reali)$. 
  Let $u$ be the weak entropy solution to~(\ref{eq:probv}) associated to $u_0\in (\L\infty\cap\L1\cap\BV)(\reali^N;\reali)$. Let us denote $u
  = u(t,x)$ and $v = u(s,y)$ for $(t,x), (s,y) \in \rpis \times
  \reali^N$. Then, for all $k,l \in \reali$ and for all test functions
  $\phi = \phi(t,x,s,y)$ in $\Cc1 \left((\rpis \times \reali^N)^2;
    \rpic \right)$, we have
  \begin{equation}
    \label{eq:utv}
    \!\!
    \begin{array}{rc}
      \displaystyle
      \!\!\int_{\rpis} \!\! \int_{\reali^N} \!\!
      \left[
        (u-k) \, \pt_t \phi
        +
        \left( f(t,x,u) - f(t,x,k) \right) \nabla_x \phi
        +
        \left(F(t,x,u)-\div f(t,x,k)\right)\phi
      \right]&
      \\
      \times
      \mathrm{sign}(u-k) \,
      \mathrm{d}x \, \mathrm{d}t
      &\geq
      0
    \end{array}
  \end{equation}
  for all $(s,y) \in \rpis \times \reali^N$, and
  \begin{equation}
    \label{eq:vtv}
    \!\!
    \begin{array}{rc}
      \displaystyle
      \!\!\!\!\!\!
      \int_{\rpis} \!\! \int_{\reali^N} \!\!
      \left[
        (v-l) \, \pt_s \phi
        +
        \left( f(s,y,v) - f(s,y,l) \right) \nabla_y \phi
        +
        (F(s,y,v)-\div f(s,y,l))\phi
      \right]&
      \\
      \times
      \mathrm{sign}(v-l) \,
      \mathrm{d}y \, \mathrm{d}s
      &\geq
      0
    \end{array}
  \end{equation}
  for all $(t,x) \in \rpis \times \reali^N$. Let $\Phi \in \Cc\infty
  (\rpis \times \reali^N;\rpic)$, $\Psi \in \Cc\infty (\reali \times
  \reali^N;\rpic)$ and set
  \begin{equation}
    \label{eq:phi}
    \varphi(t,x,s,y)=\Phi(t,x) \, \Psi(t-s,x-y) \,.
  \end{equation}
  Observe that $\pt_t \varphi + \partial_s \phi = \Psi \, \pt_t \Phi$,
  $\nabla_x \phi = \Psi \, \nabla_x\Phi + \Phi \, \nabla_x\Psi$,
  $\nabla_y \phi = -\Phi \, \nabla_x\Psi$. Choose $k = v(s,y)$
  in~(\ref{eq:utv}) and integrate with respect to
  $(s,y)$. Analogously, take $l = u(t,x)$ in~(\ref{eq:vtv}) and
  integrate with respect to $(t,x)$. Summing the obtained
  inequalities, we obtain
  \begin{equation}
    \label{eq:sumtv}
   \begin{array}{rc}
      \!\!  \displaystyle
      \int_{\rpis} \!\! \int_{\reali^N} \!\!
      \int_{\rpis} \!\! \int_{\reali^N} \!\!
      \!\!\!\!
      \mathrm{sign}(u-v)
      \bigg[
      (u-v) \, \Psi \, \pt_t \Phi +
      \left( f(t,x,u) - f(t,x,v) \right) \cdot
      \left( \nabla \Phi \right) \Psi &
      \\
      +
      \left( f(s,y,v) - f(s,y,u) - f(t,x,v) + f(t,x,u) \right) \cdot
      \left( \nabla \Psi \right) \Phi&
      \\
      +
      \left(F(t,x,u) - F(s,y,v) + \div f(s,y,u) - \div f(t,x,v) \right) \phi
      \bigg]
      \mathrm{d}x \, \mathrm{d}t \, \mathrm{d}y \, \mathrm{d}s
      &\geq
      0 .
    \end{array}
  \end{equation}
  Introduce a family of functions $\{Y_\vartheta\}_{\vartheta>0}$ such
  that for any $\vartheta>0$:\\
    \begin{figure}[htbp]
\centering
      \includegraphics[width=0.45\linewidth]{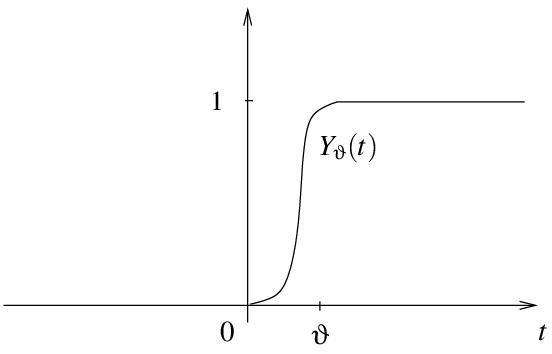}
      \includegraphics[width=0.45\linewidth]{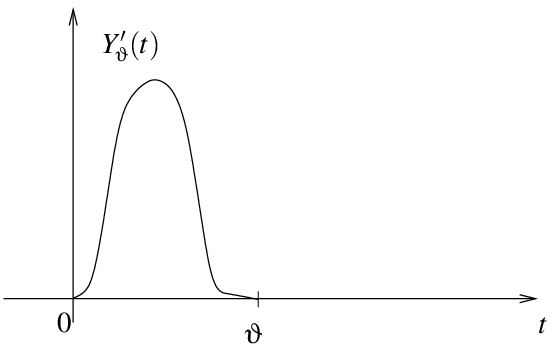}
\caption[Graphes de $Y_\vartheta$ et $Y_\vartheta'$]{Graphs of $Y_\vartheta$, left, and of $Y'_\vartheta$, right.} 
    \end{figure}
  %\end{minipage}\hfil
  %\begin{minipage}{0.45\linewidth}
    \begin{align}
        Y_\vartheta(t)
        & = 
        \displaystyle
        \int_{-\infty}^t Y_\vartheta' (s) \,\mathrm{d}s\,,
        &
        Y_\vartheta'(t)
        & = 
        \displaystyle
        \frac{1}{\vartheta} \, Y' \left( \frac{t}{\vartheta} \right)\,,
        &
        Y' & \in  \Cc\infty(\reali;\reali)\,,\label{eq:Y}
        \\
        \mathrm{Supp}(Y')
        & \subset 
        \left]0,1\right[\,,
        &
        Y'
        &\geq 
        0\,,
        &
        \displaystyle
        \int_\reali Y'(s) \, \mathrm{d}s
        & = 
        1 \,.\nonumber
    \end{align}
  %\end{minipage}

\noindent Let $T_0>0$, $\mathcal{U}=\norma{u}_{\L\infty([0,T_0]\times \reali^N;\reali)}$ and  $M = \norma{\pt_u f}_{\L\infty(\Omega^{\mathcal{U}}_{T_0}; \reali^{N})}$ which is bounded by \textbf{(H1*)}.
  Let us also define, for $\epsilon, \theta, R > 0$, $x_0 \in \reali^N$,
  (see Figure~\ref{fig:chipsi}):
  \begin{equation}
    \label{eq:chipsi}
    \chi(t) = Y_\epsilon(t)-Y_\epsilon(t-T)
    \quad \mbox{ and } \quad
    \psi(t,x) = 1-Y_\theta\left(\norma{x-x_0}-R-M(T_0-t)\right)
    \geq 0,
  \end{equation}
  where we also need the compatibility conditions $T_0 \geq T$ and $M
  \epsilon \leq R + M (T_0 - T)$.  

\begin{figure}[htbp]
    \centering
      \includegraphics[width=0.4\linewidth]{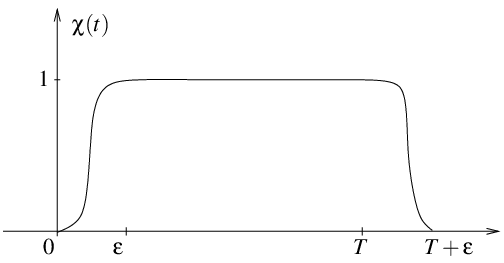}
      \includegraphics[width=0.5\linewidth]{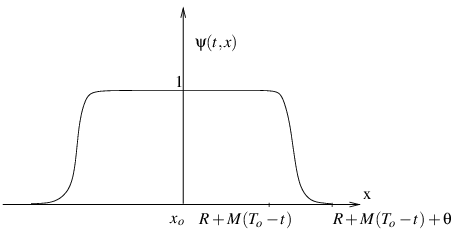}
    \caption[Graphes de $\chi$ et $\psi$]{Graphs of $\chi$, left, and of $\psi$, right.}
    \label{fig:chipsi}
  \end{figure}  
 
  Observe that $\chi \to
  \mathbf{1}_{[0,T]}$ and $\chi' \to \delta_0 - \delta_T$ as
  $\epsilon$ tends to $0$.  On $\chi$ and $\psi$ we use the bounds
  \begin{displaymath}
    \chi
    \leq
    \mathbf{1}_{[0,T+\epsilon]}
    \quad \mbox{ and } \quad
    \mathbf{1}_{B(x_0, R+M(T_0-t))}
    \leq
    \psi
    \leq
    \mathbf{1}_{B(x_0, R+M(T_0-t)+\theta)} \,.
  \end{displaymath}
  In~(\ref{eq:sumtv}), choose $\Phi(t,x) = \chi(t) \, \psi(t,x)$. With
  this choice, we have
  \begin{equation}
    \label{eq:Derivatives}
    \pt_t \Phi 
    = 
    \chi' \, \psi - M \, \chi \,  Y_\theta'
    \quad \mbox{ and } \quad
    \nabla \Phi 
    = -\chi \, Y_\theta' \, \frac{x-x_0}{\norma{x-x_0}} \,.
  \end{equation}
  Setting $\displaystyle B(t,x,u,v) = \modulo{u-v} M +
  \mathrm{sign}(u-v) \left( f(t,x,u) - f(t,x,v) \right) \cdot
  \frac{x-x_0}{\norma{x-x_0}}$, the first line in~(\ref{eq:sumtv})
  becomes
  \begin{eqnarray*}
    \!\!
    & &
    \int_{\rpis} \! \int_{\reali^N} \! \int_{\rpis} \! \int_{\reali^N} \!\!
    \left[
      (u-v) \Psi \, \pt_t \Phi +
      \! \left( f(t,x,u) - f(t,x,v) \right)
      \cdot ( \nabla \Phi) \, \Psi
    \right]\!
    \mathrm{sign}(u-v)
    \mathrm{d}x \, \mathrm{d}t \, \mathrm{d}y \, \mathrm{d}s
    \\
    \!\!
    & = &
    \int_{\rpis} \int_{\reali^N}
    \int_{\rpis} \int_{\reali^N}
    \left(
      \modulo{u-v}\, \chi' \, \psi - B(t,x,u,v) \chi \, Y_\theta'
    \right)
    \, \Psi
    \, \mathrm{d}x \, \mathrm{d}t \, \mathrm{d}y \, \mathrm{d}s
    \\
    \!\!
     &\leq &
    \int_{\rpis} \int_{\reali^N}
    \int_{\rpis} \int_{\reali^N}
    \modulo{u-v} \, \chi' \, \psi \, \Psi
    \, \mathrm{d}x \, \mathrm{d}t \, \mathrm{d}y \, \mathrm{d}s\,,
  \end{eqnarray*}
  since $B(t,x,u,v)$ is positive for all $(t, x, u, v) \in \Omega
  \times \reali$. Due to the above estimate and
  to~(\ref{eq:sumtv}), we have
  \begin{displaymath}
    \begin{array}{llcr}
      \displaystyle
      \!\!\!\!\int_{\rpis} \! \int_{\reali^N} \! \int_{\rpis} \! \int_{\reali^N} \!\!\!\!
      &
      \bigg[
      (u-v) \, \chi' \, \psi \, \Psi
      \\
      &
      \displaystyle
      +
      \left(
        f(s,y,v) - f(s,y,u) - f(t,x,v) + f(t,x,u)
      \right)
      \cdot
      (\nabla \Psi) \, \Phi
      \\
      &
      \displaystyle
      +
      \left(
        F(t,x,u) - F(s,y,v) - \div  f(t,x,v) + \div f(s,y,u)
      \right)
      \phi
      \bigg]
      \\
      &
      \displaystyle
      \times \mathrm{sign}(u-v)
      \, \mathrm{d}x\,\mathrm{d}t \, \mathrm{d}y \, \mathrm{d}s
      & \geq &
      0 .
    \end{array}
  \end{displaymath}
Now, we aim at bounds for each term of this sum. Introduce the
  following notations:
  \begin{eqnarray}
\nonumber
    I
   &  = &
    \int_{\rpis} \int_{\reali^N} \int_{\rpis} \int_{\reali^N}
    \modulo{u-v} \, \chi' \, \psi \, \Psi
    \, \mathrm{d}x\,\mathrm{d}t \, \mathrm{d}y \, \mathrm{d}s\, ,
    \\
\nonumber
    J_x
   &  = &
    \int_{\rpis} \int_{\reali^N} \int_{\rpis} \int_{\reali^N}
    \left( f(t,y,v) - f(t,y,u) + f(t,x,u) - f(t,x,v) \right)\cdot
    \left( \nabla \Psi \right) \, \Phi
    \\*  &&
\nonumber
    \qquad \qquad\qquad\qquad
    \times \,
    \mathrm{sign}(u-v)
    \, \mathrm{d}x\,\mathrm{d}t \, \mathrm{d}y \, \mathrm{d}s\, ,
    \\
\nonumber
    J_t
  &   = &
    \int_{\rpis} \int_{\reali^N} \int_{\rpis} \int_{\reali^N}
    \left( f(s,y,v) - f(s,y,u) + f(t,y,u) - f(t,y,v) \right)\cdot
    \left( \nabla \Psi \right) \, \Phi
    \\*
\nonumber
     &&
    \qquad\qquad\qquad\qquad
    \times \,
    \mathrm{sign}(u-v)
    \, \mathrm{d}x\,\mathrm{d}t \, \mathrm{d}y \, \mathrm{d}s\, ,
    \\
\nonumber
    L_1
  &  = &
    \int_{\rpis} \int_{\reali^N} \! \int_{\rpis} \int_{\reali^N}
    \!
\left[\div f(t,x,v)-\div f(t,x,u)+F(t,y,v)-F(t,y,u)\right]\!\\*
&&\qquad \qquad\qquad\qquad
    \phi \, \mathrm{sign}(u-v)
    \,\mathrm{d}x\,\mathrm{d}t\,\mathrm{d}y\,\mathrm{d}s
    \\
\nonumber    L_2
    & = &
    \int_{\rpis} \int_{\reali^N} \! \int_{\rpis} \int_{\reali^N}
    \left[
      \int_0^1 \! \nabla ( F - \div f) \left(t,rx+(1-r)y,u\right)\cdot (x-y)
      \, \mathrm{d}r
    \right]
    \phi
    \\*
\label{eq:L2}     & &
    \qquad \qquad \qquad \qquad
    \times \mathrm{sign}(u-v)
    \,\mathrm{d}x\,\mathrm{d}t\,\mathrm{d}y\,\mathrm{d}s\,.
    \\
  \nonumber  L_t
    & = &
    \int_{\rpis} \int_{\reali^N} \int_{\rpis} \int_{\reali^N}
    \left(
      F(t,y,v) - F(s,y,v) - \div  f(t,y,u) + \div f(s,y,u)
    \right) \, \varphi
    \\*
\nonumber
    & &
    \qquad\qquad\qquad\qquad
    \times \,
    \mathrm{sign}(u-v)
    \, \mathrm{d}x\,\mathrm{d}t \, \mathrm{d}y \, \mathrm{d}s \,.
  \end{eqnarray}
  Then, the above inequality is rewritten as $I + J_x+J_{t} +
  L_1+L_2+L_{t} \geq 0$.  Choose $\Psi(t,x) = \nu(t)\, \mu(x)$ where, for
  $\pd, \gd > 0$, $\mu \in \Cc\infty (\rpic; \rpic)$
  satisfies~(\ref{eq:mu})--(\ref{eq:Mu}) and
  \begin{equation}
    \label{eq:nu}
    \nu(t)
    =
    \frac{1}{\pd} \, \nu_1 \left( \frac{t}{\pd} \right)
    \,,\quad
    \int_\reali \nu_1(s) \,\mathrm{d} s
    =
    1
    \,,\quad
    \nu_1
    \in
    \Cc\infty(\reali;\rpic)
    \,,\quad
    \mathrm{supp}(\nu_1)
    \subset
    \left]-1,0\right[
    \,.
  \end{equation}

Now, we want to estimate separately $I$, $J_x$, $J_t$, $L_1$, $L_2$ and $L_t$.
Note first that if $x, y\in \reali^N\setminus \{\bigcup_{t\in [0,T_0]} \spt u(t)\} $, the integrand in $J_x$ and $L_1$ vanishes, so denoting 
\begin{equation}\label{eq:KT}
 \mathcal{S}_T(u)= \bigcup_{t\in [0,T_0]} \spt u(t)\,,
\end{equation}
the space of integration of $J_x$ and $L_1$ is in fact $\rpic\times\mathcal{S}_T(u)\times \rpic\times\mathcal{S}_T(u)$.
The main differences with respect to the proof of \cite[Theorem 2.5]{ColomboMercierRosini} are the following:
\begin{itemize}
 \item The $\L\infty$ norm that we took on $\rpic\times\reali^N\times\reali$, 
are now taken on  $\Sigma_{T_0}^u=[0,T_0]\times\mathcal{S}_T(u)\times[-\mathcal{U},\mathcal{U}]$, where $\mathcal{U}=\sup(\norma{u(t)}_{\L\infty(\reali^N)},t\in[0,T_0])$.
\item For $J_t$ and $L_t$, by Dominated Convergence Theorem, we get when $\eta\to 0$
  \begin{align}
  \lim_{\eta\to 0} J_t=\lim_{\eta\to 0} L_t =0\,,
  \end{align}
which avoids the use of time derivatives.
\item The $\L\infty$ norm of  $u$ in $L_2$ is now taken on $[-{U}_t,{U}_t]$ where ${U}_t=\norma{u(t)}_{\L\infty(\reali^N)}$.
\end{itemize}

We do not rewrite the estimates on $I$, $J_x$, $L_1$, $L_2$, that are the same as in  
\cite[Theorem 2.5]{ColomboMercierRosini}, up to the space in the $\L\infty$ norm. See remark  \ref{rem:opt} for precisions on the estimate of $L_2$.

    Letting $\epsilon, \pd, \theta \to 0$ we get
  \begin{eqnarray*}
    \limsup_{\epsilon, \pd, \theta \to 0} I
    & = &
    \int_{\reali^N} \int_{\norma{x-x_0}\leq R+MT_0}
    \modulo{u(0,x)-u(0,y)}
    \, \mu(x-y) \, \mathrm{d}x \, \mathrm{d}y
    \\*
    & &
    -
    \int_{\reali^N} \int_{\norma{x-x_0}\leq R +M(T_0-T)}
    \modulo{u(T,x)-u(T,y)}
    \, \mu(x-y) \, \mathrm{d}x \, \mathrm{d}y\, ,
    \\
    \limsup_{\epsilon, \pd, \theta \to 0} J_x
    & \leq &
    \norma{\nabla\pt_u f }_{\L\infty(\Sigma_{T_0}^u)}
    \int_0^{T}
    \int_{\reali^N}
    \int_{B(x_0, R +M(T_0-t))}
    \norma{x-y} \, \modulo{u(t,x)-u(t,y)}
    \\*
    & &
    \qquad \qquad
    \times
    \norma{\nabla \mu(x - y)}
    \,\mathrm{d}x\,\mathrm{d}y\,\mathrm{d}t\, ,
\\
    \limsup_{\epsilon, \pd, \theta \to 0} J_{t}
    & = &0\, ,
\\
    \limsup_{\epsilon, \pd, \theta \to 0} L_1
    & \leq &
    \left(
      N \norma{\nabla\pt_u f}_{\L\infty(\Sigma_{T_0}^u)}
      +
      \norma{\pt_u F}_{\L\infty(\Sigma_{T_0}^u)}
    \right)
\\*
    & &
    \qquad
    \times
    \int_0^{T} \int_{\reali^N}
    \int_{\norma{x-x_0}\leq R+M(T_0-t)}
    \modulo{u(t,x)-u(t,y)} \, \mu(x-y)
    \,\mathrm{d}x\,\mathrm{d}y\,\mathrm{d}t\, ,\\
 \limsup_{\epsilon, \pd, \theta \to 0} L_{2}
    & = & 
    \gd  M_1
    \int_0^{T} \int_{\reali^N}
    \norma{\nabla(F - \div f) (t,y,\cdot)}_{\L\infty([-{U}_t,{U}_t])}
    \, \mathrm{d}y\,\mathrm{d}t\,  ,
    \\
    \limsup_{\epsilon, \pd, \theta \to 0} L_{t}
    & = &
    0 \,,
  \end{eqnarray*}
 where
  \begin{equation}
    \label{eq:M1}
    M_1
    =
    \int_{\reali^N} \norma{x} \,
    \mu_1\left(\norma{x}\right) \, \mathrm{d}x \,.
  \end{equation}
Above, the right hand sides are bounded thanks to \textbf{(H2*)}.  

  Collating all the obtained results and using the equality,
  $\norma{\nabla \mu (x)} = - \frac{1}{ \gd^{N+1}} \mu_1'\left(
    \frac{\norma{x}}{\gd} \right)$
  \begin{equation}
    \label{inegalite}
    \begin{array}{rcl}
      & &
      \displaystyle
      \int_{\reali^N} \int_{\norma{x-x_0}\leq R +M(T_0-T)}
      \modulo{u(T,x)-u(T,y)}
      \, \frac{1}{\gd^N}
      \mu_1\left(\frac{\norma{x-y}}{\gd}\right)
      \, \mathrm{d}x \, \mathrm{d}y
      \\[10pt]
      & \leq &
      \displaystyle
      \int_{\reali^N} \int_{\norma{x-x_0}\leq R +M T_0}
      \modulo{u(0,x)-u(0,y)}
      \, \frac{1}{\gd^N}
      \mu_1\left(\frac{\norma{x-y}}{\gd}\right)
      \, \mathrm{d}x \, \mathrm{d}y
      \\[10pt]
      & &
      \displaystyle
      -
      \norma{\nabla\pt_u f }_{\L\infty(\Sigma_{T_0}^u)}
      \int_0^{T} \int_{\reali^N} \int_{\norma{x-x_0}\leq R +M(T_0-t)} \!\!\!
      \modulo{u(t,x)-u(t,y)}
      \\[10pt]
      & &
      \displaystyle
      \qquad \qquad \qquad
      \times
      \frac{1}{\gd^{N+1}} \, \mu_1'\left( \frac{\norma{x-y}}{\gd} \right)
      \, \norma{x-y}
      \,\mathrm{d}x\,\mathrm{d}y\,\mathrm{d}t
      \\[10pt]
      & &
      \displaystyle
      +
      \big(
        N \norma{\nabla\pt_u f}_{\L\infty(\Sigma_{T_0}^u)}
        +
        \norma{\pt_u F}_{\L\infty(\Sigma_{T_0}^u)}
      \big) \!
      \int_0^{T} \! \int_{\reali^N} \! \int_{\norma{x-x_0}\leq R
        +M(T_0-t)} \!
      \modulo{u(t,x)-u(t,y)}
      \\[10pt]
      & &
      \displaystyle
      \qquad \qquad \qquad
      \times
      \frac{1}{\gd^N}
      \mu_1\left(\frac{\norma{x-y}}{\gd}\right)
      \,\mathrm{d}x\,\mathrm{d}y\,\mathrm{d}t
      \\[10pt]
      & &
      \displaystyle
      +
      \gd  M_1
      \int_0^{T}\int_{\reali^N} 
      \norma{\nabla(F - \div f)(t, y,\cdot)}_{\L\infty([-{U}_t,{U}_t])} \,
      \mathrm{d}y\, \mathrm{d}t \,.
    \end{array}
  \end{equation}

  If $\norma{\nabla\pt_u f }_{\L\infty(\Sigma_{T_0}^u)} =\norma{\pt_u F}_{\L\infty(\Sigma_{T_0}^u)}=
  0$ and under the present assumption that $u_0 \in \C1
  (\reali^N;\reali)$, using Proposition~\ref{prop:normtv2},
  (\ref{eq:Cst0}) and~(\ref{eq:M1}), we directly obtain that
  \begin{equation}
    \label{eq:CK}
    \tv(u(T))
    \leq
    \tv(u_0)
    +
    \frac{M_1}{C_1}  \,
    \int_0^{T} \int_{\reali^N} 
    \norma{\nabla (F - \div f)(t, y,\cdot)}_{\L\infty([-{U}_t,{U}_t])}
    \mathrm{d}y\, \mathrm{d}t \,.
  \end{equation}
  The same procedure at the end of this proof allows to
  extend~(\ref{eq:CK}) to more general initial data, providing an
  estimate of $\tv \left( u(t) \right)$ in the situation studied
  in~\cite{bouchutperthame}.

  Now, it remains to treat the case when $\norma{\nabla \pt_u f}_{\L\infty(\Sigma_{T_0}^u)}
  \neq 0$. As in \cite[Theorem 2.5]{ColomboMercierRosini}, a direct use of Gronwall lemma is not possible, but we can first obtain an estimate of the function:
  \begin{displaymath}
    \mathcal{F}(T,\gd)
    =
    \int_0^T \int_{\reali^N} \int_{\norma{x-x_0} \leq R+M(T_0-t)}
    \modulo{u(t,x)-u(t,x-z)}
    \, \frac{1}{\gd^N}
    \, \mu_1\left(\frac{\norma{z}}{\gd}\right)
    \, \mathrm{d}x \, \mathrm{d}z \, \mathrm{d}t\,.
  \end{displaymath}
  Indeed, we get that if $T$ is such that 
\[
  T
    <
    \frac{1}{ (1+2N) \norma{\nabla\pt_u f}_{\L\infty(\Sigma_{T_0}^u)}
        + \norma{\pt_u F}_{\L\infty(\Sigma_{T_0}^u)}} \,,
\]
 then  we obtain, with $\alpha = \left( 2 N \norma{\nabla\pt_u
      f}_{\L\infty}+\norma{\pt_u F}_{\L\infty} - \frac{1}{T} \right)
  \left(\norma{\nabla\pt_u f}_{\L\infty(\Sigma_{T_0}^u)} \right)^{-1}<-1$, 
  \begin{equation}
    \label{eq:FDelta}
    \frac{1}{\gd} \mathcal{F}(T',\gd)
    \leq
    \frac{1}{-\alpha-1} \left( M_1\tv{(u_0)} +  C(T') \right)
    \frac{1}{\norma{\nabla\pt_u f}_{\L\infty(\Sigma_{T_0}^u)}} \,.
  \end{equation}
  Furthermore, by~(\ref{eq:mu}) and~(\ref{eq:Mu}) there exists a
  constant $Q > 0$ such that for all $z \in \reali^N$
  \begin{equation}
    \label{eq:BoundMu}
    -\mu_1'(\norma{z})
    \leq
    Q \, \mu_1\left(\frac{\norma{z}}{2}\right) \,.
  \end{equation}
  Divide both sides in~(\ref{inegalite}) by $\gd$, rewrite them
  using~(\ref{eq:FDelta}), (\ref{eq:BoundMu}), apply~(\ref{eq:afp})
  and obtain
  \begin{eqnarray*}
    & &
    \frac{1}{\gd}
    \int_{\reali^N} \int_{\norma{x-x_0}\leq R +M(T_0-T)}
    \modulo{u(T,x) - u(T,y)}
    \, \frac{1}{\gd^N}
    \mu_1\left(\frac{\norma{x-y}}{\gd}\right)
    \, \mathrm{d}x \, \mathrm{d}y
    \\
    & \leq &
    M_1 \tv(u_0)
    +
    \frac{ \mathcal{F}(T,\gd)}{\gd}
    \left(
      2 N \norma{\nabla\pt_u f}_{\L\infty(\Sigma_{T_0}^u)}+\norma{\pt_u F}_{\L\infty(\Sigma_{T_0}^u)}
    \right)
    \\
    & &
+
    \frac{\mathcal{F}(T,2\gd)}{2\gd}  2^{N+2}  \,Q\,
    \norma{\nabla\pt_u f}_{\L\infty(\Sigma_{T_0}^u)}\\
&&
+
    M_1
    \int_0^{T} \int_{\reali^N} 
    \norma{\nabla(F - \div f)(t, y,\cdot)}_{\L\infty([-{U}_t,{U}_t])}
    \mathrm{d}y\, \mathrm{d}t\, .
  \end{eqnarray*}
  An application of~(\ref{eq:FDelta}) yields an estimate of the type
  \begin{equation}
    \label{ineq1}
    \frac{1}{\gd} \int_{\reali^N} \int_{B(x_0,R+M(T_0-T))}
    \modulo{u(T,x)-u(T,x-z)} \, \mu(z)
    \, \mathrm{d}x \, \mathrm{d}z
    \leq \check C \,,
  \end{equation}
  where the positive constant $\check C$ is independent from $R$ and
  $\gd$. Applying Proposition~\ref{prop:normtv2} we obtain that
  $u(t)\in \BV(\reali^N;\reali)$ for $t \in \left[ 0, 2 ~T_1\right[$,
  where
  \begin{equation}
    \label{eq:T1}
    T_1
    =
    \frac{1}{2\left( (1+2N) \norma{\nabla\pt_u f}_{\L\infty(\Sigma_{T_0}^u)}
        + \norma{\pt_u F}_{\L\infty(\Sigma_{T_0}^u)}\right)} \,.
  \end{equation}

  \bigskip

  The next step is to obtain a general estimate of the $\tv$ norm. The
  starting point is~(\ref{inegalite}). Recall the
  definitions~(\ref{eq:M1}) of $M_1$ and~(\ref{eq:T1}) of
  $T_1$. Moreover, by integration by part we obtain
  \begin{displaymath}
    \int_{\reali^N} \modulo{z_1} \norma{z} \, \mu_1'(\norma{z}) \, \mathrm{d}z
    =
    - (N+1) \, C_1 \,.
  \end{displaymath}

  The following step is \emph{not} similar to \cite[proof of theorem 2.5]{ColomboMercierRosini}: we divide both terms in~(\ref{inegalite}) by $\gd$,
  apply~(\ref{eq:limm}) on the first, second and third terms  in the right hand side, with $\rho_1=\mu_1\geq 0$ in the second and third case, and with $\rho_1=-\mu_1'\geq 0$ in the second case. We obtain for
  all $T \in [0,T_1]$ with $T_1 < T_0$
  \begin{eqnarray*}
    \tv \left( u(T) \right)
    & \leq &
    \tv(u_0)
    +
    \left(
      (2N+1) \norma{\nabla\pt_u f}_{\L\infty(\Sigma_{T_0}^u)}
      +
      \norma{\pt_u F}_{\L\infty(\Sigma_{T_0}^u)}
    \right)
    \int_0^T \tv \left( u(t) \right) \, \mathrm{d}t
    \\
    & &
    +
    \frac{M_1}{C_1}  \int_0^{T} \int_{\reali^N}
    \norma{\nabla(F - \div f)(t,x,\cdot)}_{\L\infty([-{U}_t,{U}_t])} 
    \mathrm{d}x \, \mathrm{d}t\,.
  \end{eqnarray*}
  Next, an application of the Gronwall Lemma shows that $\tv \left( u(t)
  \right)$ is bounded on $[0,T_1]$
  \begin{equation}
    \label{eq:TV}
    \tv \left( u(T) \right)
    \leq
    e^{\kappa^*_0 T} \, \tv(u_0)
    +
    \frac{M_1}{C_1} \int_0^T e^{\kappa^*_0(T-t)}
    \int_{\reali^N} 
    \norma{\nabla(F-\div f)(t,x,\cdot)}_{\L\infty([-{U}_t,{U}_t])} 
    \, \mathrm{d}x\, \mathrm{d}t
  \end{equation}
  for $T \in [0,T_1]$, $M_1, C_1$ as in~(\ref{eq:M1}), (\ref{eq:Cst0})
  and $\kappa^*_0 =  (2N+1) \norma{\nabla\pt_u f}_{\L\infty(\Sigma_{T_0}^u)} +
  \norma{\pt_u F}_{\L\infty(\Sigma_{T_0}^u)} $.

Now, it remains only to relax assumption on the regularity of $u_0$ and to note that the bound (\ref{eq:TV}) is additive in time. These steps are the same as in \cite[Theorem 2.5]{ColomboMercierRosini}, so we do not write them.
\end{proofof}

\begin{remark}\label{rem:opt}
 The constant  $NW_N$ in front of  $\int_0^T\int_{\reali^N}\norma{\nabla(F-\mathrm{div}f)(t,x,\cdot)}_{\L\infty([-{U}_t,{U}_t])}\d{x}\,\d{t}$ in Theorem   \ref{teo:tv}
comes from the estimate of the term $L_2$ defined by (\ref{eq:L2}).
%  \begin{eqnarray*}
%     L_x
%     & = &
%     L_1 + L_2 \qquad \mbox{ where}
%     \\
%     L_1
%     & = &
%     \int_{\rpis} \int_{\reali^N} \! \int_{\rpis} \int_{\reali^N}
%     \!
%     \left[
%       \int_v^u \left(\pt_u \div  f(t,x,w)+\pt_u F(t,y,w) \right) \mathrm{d}w
%     \right]\!
%     \phi \, \mathrm{sign}(u-v)
%     \,\mathrm{d}x\,\mathrm{d}t\,\mathrm{d}y\,\mathrm{d}s
%     \\
%     L_2
%     & = &
%     \int_{\rpis} \int_{\reali^N} \! \int_{\rpis} \int_{\reali^N}
%     \left[
%       \int_0^1 \! \nabla ( F - \div f) \left(t,rx+(1-r)y,u\right)\cdot (x-y)
%       \, \mathrm{d}r
%     \right]
%     \phi
%     \\
%     & &
%     \qquad \qquad \qquad \qquad
%     \times \mathrm{sign}(u-v)
%     \,\mathrm{d}x\,\mathrm{d}t\,\mathrm{d}y\,\mathrm{d}s\,.
%   \end{eqnarray*}

We have indeed
\begin{eqnarray*}
L_2
&\leq &
\int_{\rpis}\int_{\reali^N} \int_{\rpis}\int_{\reali^N}\int_0^1 \modulo{ \nabla (F-\mathrm{div}f) (t,x-\gd (1-r) z, u)\cdot (\gd z)}\\
&&\qquad\qquad \qquad\times
\chi \psi \mu_1(\norma{z}) \nu \mathrm{d}r\,\mathrm{d}x\,\mathrm{d}t\, \mathrm{d}z\,\mathrm{d}s\\
&\leq&\gd \int_0^{T+\varepsilon} \int_{B(x_0, R+M(T_0-t) +\theta)} \int_{\reali^N}\int_0^1 \modulo{ \nabla (F-\mathrm{div}f) (t,x-\gd (1-r) z, u)\cdot ( z)}\\
&&\qquad\qquad\qquad\times\mu_1(\norma{z}) \mathrm{d}r\, \mathrm{d}z\,\mathrm{d}x\,\mathrm{d}t\\
&\leq&\gd \int_0^{T+\varepsilon} \int_0^1 \int_{B(x_0, R+M(T_0-t) +\theta+\gd)}\!\! \int_{\reali^N} \modulo{ \nabla (F-\mathrm{div}f) (t,x', u(t,x'+\gd (1-r)z))\cdot  z}\\
&&\qquad\qquad\qquad\times\mu_1(\norma{z})  \mathrm{d}z\,\mathrm{d}x'\,\mathrm{d}r\,\mathrm{d}t
\end{eqnarray*}
If  $F-\div f $ does not depend on $u$,  then, with the same computations as in the proof of Proposition \ref{prop:normtv2}, considering $z\mapsto  \nabla (F-\mathrm{div}f) (t,x')\cdot  z$  as a linear application, we get:
\begin{eqnarray*}
L_2
&\leq&\gd \int_0^{T+\varepsilon}  \int_{B(x_0, R+M(T_0-t) +\theta+\gd)}\modulo{ \nabla (F-\mathrm{div}f) (t,x') }\,\mathrm{d}x'\,\mathrm{d}t
\int_{\reali^N}  \modulo{z_1} \mu_1(\norma{z})  \mathrm{d}z\, ,
\end{eqnarray*}
which allows us to get rid of the constant $N W_N$ into the bound of $L_2$.

However, in the general case, because of the dependence of $u$ in $z$, we are led to take the supremum of $u(t)$. We obtain the following: 
\begin{align*}
L_2&\leq \gd \int_0^{T+\varepsilon}\!\!  \int_{B(x_0, R+M(T_0-t) +\theta+\gd)} \!\!\int_{\reali^N}\sup_{y\in \reali^N} \modulo{ \nabla (F-\mathrm{div}f) (t,x', u(t,y))\cdot z}\mu_1(\norma{z})  \mathrm{d}z\mathrm{d}x'\mathrm{d}t.
\end{align*}
We can no longer do the same computations as in the proof of  Proposition \ref{prop:normtv2}.  Indeed, it is not allowed to permute  $\sup$ and $\int_{\reali^N}$, consequently, if we want to isolate the variable $z$ from the other variables, we use the Cauchy-Schwartz inequality to obtain:
\begin{align*}
L_2\leq& \gd \int_0^{T+\varepsilon}  \int_{B(x_0, R+M(T_0-t) +\theta+\gd)}\sup_{y\in \reali^N}\norma{ \nabla (F-\mathrm{div}f) (t,x', u(t,y)) }\,\mathrm{d}x'\,\mathrm{d}t\\
&\times 
\int_{\reali^N}  \norma{z} \mu_1(\norma{z})  \mathrm{d}z\, .
\end{align*}
The constant $NW_N$ appears here when we divide by $C_1= \int_{\reali^N}  \modulo{z_1} \mu_1(\norma{z})  \mathrm{d}z$, since, by  Lemma \ref{lem:mu}, $\frac{1}{C_1}\int_{\reali^N}  \norma{z} \mu_1(\norma{z})  \mathrm{d}z =NW_N$.

In the general case, we were consequently not able, using this method, to erase the constant  $NW_N$ on the right hand side of   \Ref{result}.
\end{remark}

\begin{proofof}{Corollary \ref{cor:lip}}
 This is the same argument as in \cite[Theorem 4.3.1]{Dafermos}, the flow and the source depending here on the three variables $t$, $x$ and $u$.  
 
The weak entropy solution $u$ of (\ref{eq:probv}) is also a weak solution. Consequently, for any $\phi\in \Cc\infty([0,T]\times\reali^N,\reali)$ such that $\modulo{\phi}\leq 1$, for any $t\in [0,T]$, we have
\begin{align*}
&\int_t^T\int_{\reali^N} \left(u\pt_t\phi+f(\tau,x,u)\cdot\nabla \phi \right)\d{x}\d{\tau}+\int_{\reali^N}u(t,x)\phi(t, x)\d{x}\\
 = &  -\int_t^T\int_{\reali^N} F(\tau, x,u)\phi(\tau, x)\d{x}\d{\tau}\,.
\end{align*}
Let $s,t\in [0,T]$. Then, with $\phi(t,x)=\psi(x)$, we obtain
\begin{align*}
&\int_s^t\int_{\reali^N} f(\tau,x,u)\cdot\nabla \psi\d{x}\d{\tau}+\int_{\reali^N}(u(s,x)-u(t,x))\psi(x)\d{x}\\ 
 =& -\int_s^t\int_{\reali^N} F(\tau, x,u)\psi( x)\d{x}\d{\tau}\,.
\end{align*}
That is to say
\begin{align*}
&\int_{\reali^N}(u(s,x)-u(t,x))\psi(x)\d{x}\\ 
 =& -\int_s^t\int_{\reali^N} (F(\tau, x,u)-\div f(\tau,x,u))\psi(x)\d{x}\d{\tau}\\
 & -\int_s^t\int_{\reali^N}(\div f(\tau,x,u)\psi(x)+ f(\tau,x,u)\cdot\nabla \psi)\d{x}\d{\tau}\,.
\end{align*}
By a regularization process, we prove that 
\begin{align*}
&\modulo{\int_s^t\int_{\reali^N}(\div f(\tau,x,u)\psi(x)+ f(\tau,x,u)\cdot\nabla \psi)\d{x}\d{\tau}}\\
\leq &\modulo{s-t} \norma{\pt_u f}_{\L\infty(\Sigma_{T^u})} \sup_{[0,T]}\tv(u(t))\,.
\end{align*}
Taking the supremum over all $\psi \in \Cc\infty(\reali^N,\reali)$ such that $\modulo{\psi}\leq 1$, we obtain
\begin{align*}
\norma{u(t)-u(s)}_{\L1(\reali^N)} \leq & \modulo{\int_s^t \int_{\reali^N} \norma{F-\div f (\tau,x,\cdot)}_{\L\infty([-U_\tau,U_\tau])} \d{x}\d{\tau} }\\
&+ \modulo{s-t}\norma{\pt_u f}_{\L\infty(\Sigma_T^u])} \sup_{[0,T]}\tv(u(t))\,.
\end{align*}
\end{proofof}

\section{Proof of the stability estimates}\label{sec:proofcomparison}

We give now the proof of Theorems  \ref{teo:estimates} and \ref{thm:stab2}. 
We prove first prove the following lemma.
\begin{lemma}\label{lem:main}
Let $(f, F)$, $(g,G)$ satisfy~$\mathbf{(H1^*)}$, $(f,F)$ satisfy
  $\mathbf{(H2^*)}$ and $(f-g,F-G)$ satisfy~$\mathbf{(H3^*)}$. Let $u_0, v_0
  \in \L\infty\cap\L1\cap\BV(\reali^N; \reali)$. We denote $u$ and $v$ the solutions associated respectively to the initial conditions 
$u_0$ and $v_0$. Let $T>0$.
  Then, using the same notation as in (\ref{eq:def})--( \ref{eq:KM}), for any $R >0$ and $x_0 \in \reali^N$, the following
  estimate holds:
  \begin{eqnarray*}
      & &
    \int_{B(x_0,R+M(T_0-T))} \modulo{u(T,x)-v(T,x)} \,
    \mathrm{d}x
    \\
    & \leq & 
    \int_{B(x_0, R +MT_0)} \modulo{u(0,x)-v(0,x)}
    \,\mathrm{d}x
    \\
    & &
    +(\norma{\pt_u F}_{\L\infty( \Sigma_{T_0}^{u,v})}+\norma{\pt_u \div(g-f)}_{\L\infty(\Sigma_{T_0}^{u,v})})
    \int_0^{T} \!\! \int_{B(x_0,R+M(T_0-t))} 
    \modulo{v(t,x)-u(t,x)}
    \, \mathrm{d}x \, \mathrm{d}t
    \\
    & &
    +  \biggl[
    \int_0^T  \norma{\pt_u(f-g)(t)}_{\L{\infty}( \mathcal{S}_{T_0}(u)\times[-U_t,U_t] )} \tv(u(t))\,  \mathrm{d}t
    \\
    & &
    \qquad
    +
    \int_0^T \int_{B(x_0,R+M(T_0-t))}\!
    \norma{\left((F-G)-\div(f-g)\right)(t, y, \cdot)}_{\L{\infty}([-{V}_t, {V}_t])}
    \mathrm{d}y\, \mathrm{d}t
    \biggr]\,.
  \end{eqnarray*}
\end{lemma}

The beginning of this proof is similar, up to (\ref{eq:I}), to the proof of Theorem 2.6 in \cite{ColomboMercierRosini}. We rewrite it in order to be complete and clear.

\begin{proofof}{Lemma~\ref{lem:main}}

  Let $\Phi\in \Cc\infty( \rpis \times \reali^N; \rpic)$, $\Psi \in
  \Cc\infty (\reali\times \reali^N; \rpic)$, and set
  $\varphi(t,x,s,y)=\Phi(t,x) \Psi(t-s,x-y)$ as in~(\ref{eq:phi}).

By Kru\v zkov Theorem \cite[Theorem 5 \&  Section  5, Remark 4]{Kruzkov}, the set of hypotheses \textbf{(H1*)} gives us existence and uniqueness of a weak entropy solution for any initial condition in $\L\infty\cap\L1(\reali^N;\reali)$. Let  $u$ be the Kru\v zkov solution associated to $u_0$ and $v$ be the Kru\v zkov solution associated to $v_0$. 
  By definition of Kru\v zkov weak entropy solution, we have for all $ l\in \reali$, for all $
  (t,x) \in \rpis \times \reali^N$
  \begin{equation}
    \label{eq:u}
    \begin{array}{r}
      \displaystyle
      \!\!\int_{\rpis} \!\! \int_{\reali^N} \!\!
      \left[
        (u-l) \, \pt_s \phi
        +
        \left( f(s,y,u) - f(s,y,l) \right) \cdot \nabla_y \phi
        +
        \left(F(s,y,u)-\div f(s,y,l)\right) \phi
      \right]
      \\
      \times \mathrm{sign}(u-l) \,
      \mathrm{d}y \, \mathrm{d}s
      \geq
      0
    \end{array}
  \end{equation}
  and for all $ k\in \reali$, for all $ (s,y)\in \rpis \times \reali^N$
  \begin{equation}
    \label{eq:v}
    \begin{array}{r}
      \displaystyle
      \!\!\int_{\rpis} \!\! \int_{\reali^N} \!\!
      \left[
        (v-k) \, \pt_t \phi
        +
        \left( g(t,x,v) - g(t,x,k) \right) \cdot \nabla_x \phi
        +
        \left(G(t,x,v) - \div g(t,x,k)\right)\phi
      \right]
      \\
      \times \mathrm{sign}(v-k) \,
      \mathrm{d}x \, \mathrm{d}t
      \geq
      0 .
    \end{array}
  \end{equation}
  Choose $k=u(s,y)$ in~(\ref{eq:v}) and integrate with respect to
  $(s,y)$. Analogously, take $l=v(t,x)$ in~(\ref{eq:u}) and integrate
  with respect to $(t,x)$. By summing the obtained equations, we get,
  denoting $u=u(s,y)$ and $v=v(t,x)$:
  \begin{equation}
    \label{eq:sum}
    \!\!\!
    \begin{array}{rl@{\;}c@{\;}l}
      \displaystyle 
      \int_{\rpis} \int_{\reali^N} \int_{\rpis} \int_{\reali^N} 
      &
      \displaystyle
      \bigg[
      (u-v) \Psi\pt_t \Phi 
      +
      \left(g(t,x,u)-g(t,x,v)\right) \cdot \left(\nabla\Phi\right) \Psi
      \\
      &
      +
      \left( g(t,x,u) - g(t,x,v) - f(s,y,u) + f(s,y,v) \right)
      \cdot
      (\nabla\Psi) \Phi
      \\
      &
      + 
      \left( 
        F(s,y,u) - G(t,x,v) + \div  g(t,x,u) -\div f(s,y,v)
      \right)
      \varphi 
      \bigg]
      \\
      &
      \times \mathrm{sign}(u-v)
      \, \mathrm{d}x\,\mathrm{d}t\,\mathrm{d}y\,\mathrm{d}s 
      & \geq & 
      0 \,.
    \end{array}
  \end{equation}
We introduce a family of functions $\{Y_\vartheta\}_{\vartheta>0}$ as in~(\ref{eq:Y}). Let $T_0>0$ and  denote  $M = \norma{\pt_u
    g}_{\L\infty(\Omega^{\mathcal{V}}_{T_0};\reali^{N})}$ with $\mathcal{V}=\max(\norma{u}_{\L\infty([0,T_0]\times\reali^N)}, \norma{v}_{\L\infty([0,T_0]\times\reali^N)})$. We also define $\chi, \psi$ as
  in~(\ref{eq:chipsi}), for $\epsilon, \theta, R > 0$, $x_0 \in
  \reali^N$ (see also Figure~\ref{fig:chipsi}). Note that with
  these choices, equalities~(\ref{eq:Derivatives}) still hold. Note
  that here the definition of the test function $\phi$ is essentially
  the same as in the preceding proof; the only change is the
  definition of the constant $M$, which is now defined with reference
  to $g$. We also introduce as above the function $\displaystyle
  B(t,x,u,v) = M \modulo{u-v} + \mathrm{sign}(u-v) \left( g(t,x,u) -
    g(t,x,v) \right) \cdot \frac{x-x_0}{\norma{x-x_0}}$ that is
  positive for all $(t,x,u,v) \in \Omega \times \reali$, and we
  have:
  \begin{align*}
     & 
    \int_{\rpis} \! \int_{\reali^N} \! \int_{\rpis} \! \int_{\reali^N}
    \left[ 
      (u-v) \pt_t \Phi 
      + 
      \left( g(t,x,u) - g(t,x,v) \right) \cdot \nabla \Phi
    \right]
    \Psi
    \, \mathrm{sign}(u-v) \, \mathrm{d}x \, \mathrm{d}t \, \mathrm{d}y
    \, \mathrm{d}s 
    \\
     \leq &
    \int_{\rpis} \! \int_{\reali^N} \! \int_{\rpis} \! \int_{\reali^N}
    \left[
      \modulo{u-v} \chi' \psi
      -
      B(t,x,u,v) \chi Y_\theta'
    \right]
    \Psi
    \, \mathrm{d}x \, \mathrm{d}t \, \mathrm{d}y \, \mathrm{d}s
    \\
     \leq &
    \int_{\rpis} \! \int_{\reali^N} \! \int_{\rpis} \! \int_{\reali^N}
    \modulo{u-v} \, \chi' \, \psi \, \Psi
    \, \mathrm{d}x \, \mathrm{d}t \, \mathrm{d}y \, \mathrm{d}s \,.
  \end{align*}
  Due to the above estimate and~(\ref{eq:sum}), we obtain
  \begin{displaymath}
    \begin{array}{llcr}
      \displaystyle
      \int_{\rpis} \int_{\reali^N} \int_{\rpis} \int_{\reali^N} 
      &
      \bigg[(u-v) \chi'\psi\Psi 
      & &
      \\
      &
      \displaystyle
      +
      \left( g(t,x,u) - g(t,x,v) - f(s,y,u) + f(s,y,v) \right)
      \cdot
      (\nabla\Psi) \Phi 
      & &
      \\
      &
      \displaystyle 
      + 
      \left( 
        F(s,y,u) - G(t,x,v) + \div  g(t,x,u) - \div f(s,y,v)
      \right)
      \varphi
      \bigg]
      & &
      \\
      &
      \displaystyle
      \times \mathrm{sign}(u-v)
      \,\mathrm{d}x \,\mathrm{d}t\,\mathrm{d}y\,\mathrm{d}s 
      & \geq & 
      0 \,,
    \end{array}
  \end{displaymath}
  i.e.~$I + J_x + J_t + K + L_x + L_t \geq 0$, where
  \begin{align}
    \label{eq:I}
    I
     = &
    \int_{\rpis} \int_{\reali^N} \int_{\rpis} \int_{\reali^N} 
    \modulo{u-v} \chi'\psi\Psi 
    \,\mathrm{d}x\,\mathrm{d}t\,\mathrm{d}y\,\mathrm{d}s\, , 
    \\
   \nonumber
    J_x
     = & 
    \int_{\rpis} \int_{\reali^N} \int_{\rpis} \int_{\reali^N}
   \big[ \left(g(t,x,u)-g(t,x,v)+g(t,y,v)-g(t,y,u) \right)
    \cdot (\nabla \Psi) \Phi 
    \\
   \label{eq:Jx}
     & 
    \qquad \qquad \qquad \qquad
    +(\div g(t,x,u)-\div g(t,x,v))\phi\big] \mathrm{sign}(u-v)
    \, \mathrm{d}x \, \mathrm{d}t \, \mathrm{d}y \, \mathrm{d}s \, ,
    \\
    \nonumber
    J_t
     = &
    \int_{\rpis} \int_{\reali^N} \int_{\rpis} \int_{\reali^N}
    \big[\left( f(s,y,v) - f(s,y,u) + f(t,y,u) - f(t,y,v) \right)
    \cdot (\nabla \Psi) \Phi
    \\
    \nonumber
     &
    \qquad \qquad \qquad \qquad +(\div f(t,y,v)-\div f (s,y,v))\big]
    \times \mathrm{sign}(u-v)
    \,\mathrm{d}x\,\mathrm{d}t\,\mathrm{d}y\,\mathrm{d}s\, ,
    \\
    \label{eq:K}
    K
     = &
    \int_{\rpis} \int_{\reali^N} \int_{\rpis} \int_{\reali^N}
    \left( (g-f)(t,y,u) - (g-f)(t,y,v) \right)
    \cdot (\nabla\Psi)  \Phi
    \\
    \nonumber
     &
    \qquad \qquad \qquad \qquad
    \times 
    \mathrm{sign}(u-v)
    \,\mathrm{d}x\,\mathrm{d}t\,\mathrm{d}y\,\mathrm{d}s\,
    ,
    \\
    \nonumber
    L_x
     = &
    \int_{\rpis} \int_{\reali^N} \int_{\rpis} \int_{\reali^N}
    \left( 
      F(t,y,u) - G(t,x,v) + \div g(t,x,v) - \div f(t,y,v)
    \right)
    \varphi 
    \\
    \label{eq:Lx}
     &
    \qquad \qquad \qquad \qquad
    \times \mathrm{sign}(u-v)
    \,\mathrm{d}x\,\mathrm{d}t\,\mathrm{d}y\,\mathrm{d}s\, ,
    \\
    \nonumber
    L_t
    = &
    \int_{\rpis} \int_{\reali^N} \int_{\rpis} \int_{\reali^N}
    \left( 
      F(s,y,u) - F(t,y,u) 
    \right)
    \varphi 
    \\
    \nonumber
     &
    \qquad \qquad \qquad \qquad
    \times \mathrm{sign}(u-v) 
    \,\mathrm{d}x\,\mathrm{d}t\,\mathrm{d}y\,\mathrm{d}s
    \,.
  \end{align}
  Now, we choose $\Psi(t,x) = \nu(t) \, \mu(x)$ as
  in~(\ref{eq:nu}),~(\ref{eq:mu}),~(\ref{eq:Mu}).  Let us estimate each of these integrals separately.

\paragraph{Estimate on $I$.}  
  The estimate on $I$ is the same as in the proof of \cite[Theorem 2.6]{ColomboMercierRosini}: thanks to
  Lemma~\ref{lem:estI},  we obtain
  \begin{eqnarray}
    \label{eq:estI}
    \limsup_{\epsilon, \pd,\gd \to 0} I 
    & \leq & 
    \int_{\norma{x-x_0} \leq R + MT_0+\theta} \modulo{u(0,x)-v(0,x)} \,\mathrm{d}x\\
    \nonumber
    & &
    -
    \int_{\norma{x-x_0}\leq R +M(T_0-T)} \modulo{u(T,x)-v(T,x)} \, \mathrm{d}x\, .
    \end{eqnarray}

\paragraph{Estimate on $J_x$.}  
For $ J_x$, we derive a new estimate with respect to  \cite[Theorem 2.6]{ColomboMercierRosini}. Indeed,   as $g $ is $\C2$ in space, we can use the following Taylor expansion:
\begin{align*}
g(t,y,v)=&g(t,x,v)+ \nabla g (t,x,v)\cdot(y-x) \\
&+\int_0^1 (1-r) \nabla^2 g(t,r y+(1-r)x,v)\d{r} \cdot(y-x)^2\,,\\
g(t,y,u)=&g(t,x,u)+ \nabla g (t,y,u)\cdot(y-x) \\
&+\int_0^1 (1-r) \nabla^2 g(t,r y+(1-r)x,u)\d{r} \cdot(y-x)^2\,.
\end{align*}
Besides, we note that
\begin{eqnarray*}
&&\left(\nabla g (t,x,v)\cdot(y-x)\right)\cdot \nabla \mu(x-y) -\div g(t,x,v) \,\mu(x-y)\\
&=&\sum_{i,j} \pt_{j} g_i(t,x,v) \,(y_j-x_j)\, \pt_i\mu(x-y) -\sum_i\pt_i g_i(t,x,v)\, \mu(x-y)\\
&=& -\sum_{i,j}\pt_j g_i(t,x,v) \,\pt_i\left(z_j\mu(z)\right)|_{z=x-y} \\
&=&-\nabla g(t,x,v)\cdot\nabla((x-y)\mu(x-y))
\end{eqnarray*}
In the same way, we have
\begin{eqnarray*}
&&\left(\nabla g(t,x,u)\cdot(x-y)\right)\nabla \mu(x-y) +\div g(t,x,u)\mu(x-y)\\
&=&\nabla g(t,x,u)\cdot\nabla((x-y)\mu)
\end{eqnarray*}
so that finally
\begin{align*}
&\left(g(t,y,v)-g(t,x,v)+g(t,x,u)-g(t,y,u)\right)\nabla \mu +\left(\div g(t,x,u)-\div g(t,x,v)\right)\mu(x-y)\\
&=\left(\nabla g(t,x,u)-\nabla g(t,x,v)\right)\cdot\nabla((x-y)\mu)\\
&\quad+ \left[\int_0^1 (1-r) \big(\nabla^2 g(t,r y+(1-r)x,u)-\nabla^2 g(t,ry +(1-r)x, v)\big)\d{r} \cdot(x-y)^2\right]\cdot\nabla \mu 
\end{align*}
After  a change of variable, we obtain 
\begin{align*}
&\lim_{\varepsilon,\eta,\theta\to 0}J_x\\
=&\int_0^T\!\!\int_{B(x_0,R+M(T_0-t))} \int_{\reali^N}\!\!\Big\{ \big(\nabla g(t,x,u(t,x-\lambda z))-\nabla g(t,x,v(t,x))\big)\\
&\qquad\qquad\qquad\qquad\qquad\cdot\nabla(z\mu_1(\norma{z}) )\mathrm{sign}(u-v) \\
&\qquad\qquad+ \lambda \Big[\int_0^1 (1-r)\big(\nabla^2 g(t,r y+(1-r)x,u) -\nabla^2 g(t,ry +(1-r)x, v)\big)\d{r}\cdot z^2\Big]\\
&\qquad\qquad\qquad\qquad\qquad\cdot \frac{z}{\norma{z}}\mu_1'(\norma{z})\,\mathrm{sign}(u-v) \Big\} \d{z}\d{x}\d{t}\,.
\end{align*}
When $\lambda$ goes to 0, we obtain by the Dominated Convergence Theorem
\begin{align*}
\lim_{\varepsilon,\eta, \theta, \lambda \to 0} J_x=&\int_0^T \int_{B(x_0,R+M(T_0-t))} \left(\nabla g(t,x,u(t,x))-\nabla g(t,x,v(t,x))\right)\, \mathrm{sign}(u-v)\d{x}\d{t} \\
&\qquad\qquad\qquad\cdot\int_{\reali^N}\nabla(z\mu_1(\norma{z}) )\d{z}\,.
\end{align*}
As $\int_{\reali^N}\nabla(z\mu_1(\norma{z}) )\d{z}=0$, we finally get
\begin{align}  
\label{eq:estJ} 
\lim_{\varepsilon,\eta, \theta, \lambda \to 0} J_x=&0\,.
\end{align}

\paragraph{Estimates of $J_t$ and $L_t$.}
For $J_t$ and $L_t$, we avoid now the use of the derivatives in time thanks to an application of  the Dominated Convergence Theorem. We obtain
\begin{align}
\lim_{\varepsilon,\eta, \theta, \lambda \to 0} J_t=\lim_{\varepsilon,\eta, \theta, \lambda \to 0} L_t=&0\,.\label{eq:est_Jt}
\end{align}

\paragraph{Estimate of $L_x$.}
For $L_x$, we have
\begin{align*}
 L_x\leq&\int_{\rpis} \int_{\reali^N} \int_{\rpis} \int_{\reali^N} \big[(F-G-\div(f-g))(t,y,v)+( F(t,y,u)-F(t,y,v))\\
&\qquad\qquad +\int_0^1 \nabla G (t, ry+(1-r)x,v)\cdot (y-x)\d{r}\big]\phi\d{y}\,\d{s}\,\d{x}\,\d{t}\,.
\end{align*}

Note that $F(t,y,u)-F(t,y,v)=\int_v^u \pt_u F(t,y,w) \d{w}$ vanishes for $y\in \reali^N\setminus \mathcal{S}_T(u,v)$. 
Consequently, with $\mathcal{V}=\sup_{t\in [0,T_0]}(\norma{u(t)}_{\L\infty(\reali^N)}, \norma{v(t)}_{\L\infty(\reali^N)})$ and  
\begin{equation}\label{KTUV}
 \Sigma_{T_0}^{u,v}=[0,T_0]\times \mathcal{S}_T(u,v) \times [-\mathcal{V},\mathcal{V}]\,,
\end{equation}
we obtain
\begin{align}
\lim_{\varepsilon, \eta, \theta,\lambda\to 0} L_x\leq &\int_0^T \int_{B(x_0, R+M(T_0-t))}\norma{(F-G)(t,x,\cdot) - \div(f-g)(t,x,\cdot) }_{\L\infty([-{U}_t,{U}_t])}\d{x}\d{t} \nonumber \\
&+\norma{\pt_u F}_{\L\infty( \Sigma_{T_0}^{u,v})}\int_0^T\int_{B(x_0,R+M(T_0-t))} \modulo{u(t,x)-v(t,x)}\d{x}\d{t}\,.\label{eq:est_Lx}
\end{align}

\paragraph{Estimate of $K$.}
  In order to estimate $K$ as given in~(\ref{eq:K}), we  follow the same procedure as in  \cite[Theorem 2.6]{ColomboMercierRosini}: let us introduce a
  regularisation of the $y$ dependent functions. In fact, let
  $\rho_\alpha(z) = \frac{1}{\alpha} \rho
  \left(\frac{z}{\alpha}\right)$ and $\sigma_\beta(y) =
  \frac{1}{\beta^N} \sigma \left(\frac{y}{\beta}\right)$, where $\rho
  \in \Cc\infty( \reali; \rpic)$ and $\sigma \in \Cc{\infty}(
  \reali^N; \rpic)$ are such that $\norma{\rho}_{\L1(\reali;\reali)} =
  \norma{\sigma}_{\L1(\reali^N;\reali)} = 1$ and $\mathrm{Supp} (\rho)
  \subseteq \left]-1,1\right[$, $\mathrm{Supp} (\sigma) \subseteq
  B(0,1)$. Then, we introduce
  \begin{displaymath}
    \begin{array}{rcl@{\qquad}rcl}
      P(w)
      & = & 
      (g-f)(t,y,w)\,,
      &
      s_\alpha
      & = &
      \mathrm{sign} \ast_u \rho_\alpha \,,
      \\
      \Upsilon_\alpha^i (w)
      & = &
      s_\alpha(w-v) \, \left( P_i(w) - P_i(v) \right) ,
      &
      u_\beta 
      & = & 
      \sigma_\beta \ast_y u \,,
      \\
      \Upsilon^i (w)
      & = &
      \mathrm{sign} (w-v) \, \left( P_i(w) - P_i(v) \right) ,
    \end{array}
  \end{displaymath}
  so that we obtain
  \begin{eqnarray*}
    & &
    \langle 
    \Upsilon_\alpha^i(u_{\beta}) - \Upsilon_\alpha^i(u), \, \pt_{y_i} \varphi 
    \rangle 
    \\
    & = &\!\!\int_{\reali^N} \left[\left( s_\alpha(u-v)P_i(u)-s_\alpha(u_\beta-v)P_i(u_\beta)\right)+\left( s_\alpha(u-v)-s_\alpha(u_\beta-v)\right) P_i(v) \right]\pt_{y_i}\phi \d{y}\\
&=& \!\!\int_{\reali^N}\int_u^{u_\beta} \left(\pt_U(s_\alpha(U-v)P_i(U))-\pt_U s_\alpha(U-v)P_i(v)\right)\pt_{y_i}\phi \d{y}\\
&=&\!\! \int_{\reali^N} \int_{u}^{u_\beta}\left( s_\alpha'(U-v)(P_i(U)-P_i(v))+s_\alpha(U-v)P_i'(U)\right) \pt_{y_i}\phi \d{y}\,.
  \end{eqnarray*}
  Now, we use the relation $s_\alpha'(U) = \frac{2}{\alpha} \rho
  \left( \frac{U}{\alpha} \right)$ to obtain
  \begin{eqnarray*}
    \modulo{
      \langle 
      \Upsilon_\alpha^i(u_{\beta}) - \Upsilon_\alpha^i(u), \,
      \pt_{y_i} \varphi 
      \rangle }
    & \leq & 
    \int_{\reali^N} \int_{\reali}{2}
      \rho \left( z \right) \modulo{P_i(v+\alpha z)-P_i(v)} \, \pt_{y_i} \varphi 
    \d{z}\,\mathrm{d}y
    \\
    & &
    +
    \int_{\reali^N} \int_{u}^{u_\beta} \modulo{P_i'(U)} \pt_{y_i} \varphi 
    \, \mathrm{d}U \, \mathrm{d}y \,.
  \end{eqnarray*}
  When $\alpha$ tends to $0$, using  the Dominated Convergence
  Theorem we obtain
  \begin{eqnarray*}
    \modulo{
      \langle 
      \Upsilon^i(u_{\beta})-\Upsilon^i(u), \pt_{y_i}\varphi
      \rangle 
    }
    & \leq &
    \int_{\reali^N} \modulo{u-u_\beta} \, \norma{P_i'}_{\L\infty}
    \pt_{y_i} \varphi \, \mathrm{d}y\,.
  \end{eqnarray*}
  Applying the Dominated Convergence Theorem again, we see that
  \begin{eqnarray*}
    \lim_{\beta \to 0} \; \lim_{\alpha \to 0}
    \langle \Upsilon_\alpha^i(u_{\beta}), \, \pt_{y_i} \varphi \rangle
    & = &
    \langle \Upsilon^i(u), \, \pt_{y_i} \varphi \rangle \,,
    \\
    \lim_{\beta \to 0} \; \lim_{\alpha \to 0}
    \langle \Upsilon_\alpha(u_{\beta}), \, \nabla_y\varphi \rangle 
    & = &
    \langle \Upsilon(u),\, \nabla_y\varphi \rangle \,.
  \end{eqnarray*}
  Consequently, it is sufficient to find a bound independent of
  $\alpha$ and $\beta$ on $K_{\alpha,\beta}$, where
  \begin{displaymath}
    K_{\alpha, \beta}
    =
    - \int_{\rpis} \int_{\reali^N} \int_{\rpis} \int_{\reali^N} 
    \Upsilon_\alpha(u_\beta) \cdot \nabla_y \varphi
    \, \mathrm{d}x \, \mathrm{d}t \, \mathrm{d}y \, \mathrm{d}s \,.
  \end{displaymath}
  Integrating by parts, we obtain
  \begin{align*}
    K_{\alpha,\beta}
     = &
    \int_{\rpis} \int_{\reali^N} \int_{\rpis} \int_{\reali^N} \!\!
    \Div_{y} 
    \Upsilon_\alpha (u_\beta) \,\varphi 
    \,\mathrm{d}x \, \mathrm{d}t \, \mathrm{d}y \, \mathrm{d}s
    \\
     = &
    \int_{\rpis} \!\! \int_{\reali^N}\!\! \int_{\rpis}\!\! \int_{\reali^N}  \!\!
    \pt_u s_\alpha(u_\beta-v)  \nabla u_\beta \cdot
    \left( (g-f) (t,y,u_\beta) - (g-f) (t,y,v) \right) \phi
    \, \mathrm{d}x \,\mathrm{d}t\,\mathrm{d}y\,\mathrm{d}s 
    \\
     &
    +
    \int_{\rpis}\!\! \int_{\reali^N}\!\! \int_{\rpis} \!\!\int_{\reali^N} \!\!
    s_\alpha (u_\beta-v) 
    \left( \pt_u (g-f) (t,y,u_\beta) \cdot \nabla u_\beta \right) \phi 
    \, \mathrm{d}x \, \mathrm{d}t \, \mathrm{d}y \, \mathrm{d}s
    \\
    &+ \int_{\rpis} \!\!\int_{\reali^N}\!\! \int_{\rpis} \!\!\int_{\reali^N} \!\!
    s_\alpha (u_\beta-v) 
    \left( \div (g-f) (t,y,u_\beta) -\div(g-f)(t,y,v) \right) \phi 
    \, \mathrm{d}x \, \mathrm{d}t \, \mathrm{d}y \, \mathrm{d}s
    \\
     = &
    K_1 + K_2 +K_3 \,.
  \end{align*}
  We now search for a bound for each term of the sum above.
  \begin{itemize}
  \item For $K_1$, recall that $\pt_u s_\alpha(u) = \frac{2}{\alpha}
    \rho \left( \frac{u}{\alpha} \right)$. Hence, by the Dominated
    Convergence Theorem, we get that $K_1\to 0$ when $\alpha\to
    0$. Indeed,
    \begin{eqnarray*}
      & &
      \modulo{\frac{2}{\alpha} \rho
        \left( \frac{u_\beta-v}{\alpha} \right) \, \nabla u_\beta 
        \cdot
        \left( (g-f)(t,y,u_\beta) - (g-f)(t,y,v) \right) \, \varphi} 
      \\
      & \leq &
      \frac{2}{\alpha} \rho\left( \frac{u_\beta-v}{\alpha} \right)
      \varphi
      \norma{\nabla u_\beta(s,y)}
      \int_{v}^{u_\beta} \norma{\pt_u(f-g)(t,y,w)}
      \, \mathrm{d}w
      \\
      & \leq & 
      2\norma{\rho}_{\L\infty(\reali;\reali)} \, \norma{\nabla u_\beta(s,y)} \,
      \norma{\pt_u(f-g)}_{\L\infty(\Omega_U^{T_0};\reali^N)} \, \varphi 
      \qquad 
      \in \L1 \left((\rpis\times\reali^N)^2;\reali\right).
    \end{eqnarray*}
  \item For $K_2$,  denoting $\mathcal{D}=\{\mathcal{S}_{T_0}(u)+B(0,\beta)\}\times[-\norma{u(t)}_{\L\infty},\norma{u(t)}_{\L\infty}]$, we get
    \begin{eqnarray*}
      K_2
      & \leq &
      \int_0^{T+\epsilon+\pd} \int_{\reali^N} \int_{\rpic} \norma{\pt_u(f-g)(t)}_{\L{\infty}( \mathcal{D}; \reali^N)}  
      \norma{\nabla u_\beta(s,y)} \nu(t-s)
      \, \mathrm{d}y \, \mathrm{d}s\d{t}
      \\
      & \leq &
      \int_0^{T+\varepsilon+\pd} \int_{\rpic} \norma{\pt_u(f-g)(t)}_{\L{\infty}( \mathcal{D}; \reali^N)} \, \tv(u_\beta(s))\, \nu(t-s)\d{s}\mathrm{d}t\,.
    \end{eqnarray*}
We note besides that $\mathcal{D}\to \mathcal{S}_{T_0}(u)\times[-U_t,U_t]$ when $\beta\to 0$.
    \item For $K_3$, we have
    \[
\lim_{\alpha,\beta,\epsilon, \pd, \gd \to 0}K_3\leq     \int_{0}^T \int_{B(x_0, R+M(T_0-t))}\norma{\pt_u \div(g-f)}_{\L\infty(\Sigma_{T_0}^{u,v})} \modulo{(u-v)(t,x)}\d{x}\d{t}\,.
   \]
  \end{itemize}
  Finally, letting $\alpha,\beta\to 0$ and $\epsilon, \pd, \gd \to 0$,
  due to \cite[Proposition~3.7]{AmbrosioFuscoPallara}, we obtain
  \begin{equation}
    \label{eq:estK}
    \begin{array}{rcl}
      \displaystyle   
      \limsup_{\epsilon, \pd, \gd \to 0}  K 
      & \leq &
      \displaystyle 
    \int_0^T    \norma{\pt_u(f-g)(t)}_{\L{\infty}( \mathcal{S}_{T_0}(u)\times[-U_t,U_t] ;\reali^N)} \tv(u(t))\, \mathrm{d}t \\
    && \displaystyle  +\int_{0}^T \int_{B(x_0, R+M(T_0-t))}\norma{\pt_u \div(g-f)}_{\L\infty(\Sigma_{T_0}^{u,v})} \modulo{(u-v)(t,x)}\d{x}\d{t}\,.
    \end{array}
  \end{equation}

\paragraph{Collecting of the estimates.}
  Now, we collate the estimates obtained in~(\ref{eq:estI}),
  (\ref{eq:estJ}), (\ref{eq:est_Lx}), 
  and (\ref{eq:estK}). Remark the order in which we pass to the
  various limits: first $\epsilon,\pd, \theta\to 0$ and, after, $\gd
  \to 0$. Therefore, we get
  \begin{eqnarray*}
    & &
    \int_{B(x_0,R+M(T_0-T))} \modulo{u(T,x)-v(T,x)} \,
    \mathrm{d}x
    \\
    & \leq & 
    \int_{B(x_0, R +MT_0)} \modulo{u(0,x)-v(0,x)}
    \,\mathrm{d}x
    \\
    & &
    +(\norma{\pt_u F}_{\L\infty( \Sigma_{T_0}^{u,v})}+\norma{\pt_u \div(g-f)}_{\L\infty(\Sigma_{T_0}^{u,v})})
    \int_0^{T} \!\! \int_{B(x_0,R+M(T_0-t))} 
    \modulo{v(t,x)-u(t,x)}
    \, \mathrm{d}x \, \mathrm{d}t
    \\
    & &
    +  \biggl[
    \int_0^T  \norma{\pt_u(f-g)(t)}_{\L{\infty}( \mathcal{S}_{T_0}(u)\times[-U_t,U_t] )} \tv(u(t))\,  \mathrm{d}t
    \\
    & &
    \qquad
    +
    \int_0^T \int_{B(x_0,R+M(T_0-t))}\!
    \norma{\left((F-G)-\div(f-g)\right)(t, y, \cdot)}_{\L{\infty}([-{V}_t, {V}_t])}
    \mathrm{d}y\, \mathrm{d}t
    \biggr]\,.
  \end{eqnarray*}
\end{proofof}  

\begin{remark}
In the preceding proof, the main changes comparing to \cite{ColomboMercierRosini} are essentially in the bound of $J_x$.  Furthermore, we also gain some regularity hypotheses by avoiding the use of the derivative in time.
\end{remark}

  \begin{proofof}{Theorem \ref{teo:estimates}}
Thanks to Lemma \ref{lem:main}, we can write 
  \begin{equation}\label{eq:aaa}
    A'(T) \leq A'(0) + \kappa^* \, A(T) + R(T) \,,
  \end{equation}
  where
  \begin{eqnarray}
    \nonumber
    A(T)
    & = &
    \int_0^{T} \int_{B(x_0, R +M(T_0-t))}
    \modulo{v(t,x)-u(t,x)}
    \, \mathrm{d}x \, \mathrm{d}t\, ,
    \\
    \label{eq:table}
    \kappa^*
    & = &
    \norma{\pt_u F }_{\L{\infty}( \Sigma_{T_0}^{u,v})}+\norma{\pt_u \div(g-f)}_{\L\infty(\Sigma_{T_0}^{u,v})}\,,
    \\
    \nonumber
    R(T)
    & = &
    \norma{\pt_u(f-g)}_{\L{\infty}( \Sigma_{T_0}^{u})}
    \int_0^T \tv\left(u(t) \right)\, \mathrm{d}t
    \\
 \nonumber   & &
    \quad
    +
    \int_0^T \int_{B(x_0,R+M(T_0-t))}\!
    \norma{\left((F-G)-\div(f-g)\right)(t,y,\cdot)}_{\L{\infty}([-{V}_t,{V}_t])}
    \!
    \mathrm{d}y \, \mathrm{d}t .
  \end{eqnarray}
  The bound~(\ref{result}) on $\tv\left(u(t)\right)$ gives:
  \begin{eqnarray*}
    R(T)
    & \leq & 
    \frac{e^{\kappa^*_0 T} -1}{\kappa^*_0} a
    +
    \int_0^T \frac{e^{\kappa^*_0 (T-t)} -1}{\kappa^*_0} b(t)\mathrm{d}t 
    +
    \int_0^T c(t)\mathrm{d}t\,,
  \end{eqnarray*}
  where $\kappa^*_0$ is defined in~(\ref{eq:kappao}) and
  \begin{eqnarray*}
    a
    & = &
    \norma{\pt_u(f-g)}_{\L{\infty}( \Sigma_{T_0}^{u} )}\tv(u_0) \,,
    \\
    b(t)
    & = & 
    N W_N \norma{\pt_u(f-g)}_{\L{\infty}( \Sigma_{T_0}^{u})} 
    \int_{\reali^N}
    \norma{\nabla(F-\div  f)(t,x,\cdot)}_{\L\infty([-{U}_t,{U}_t])}
    \, \mathrm{d}x \,,
    \\
    c(t)
    & = &
    \int_{B(x_0,R+M(T_0-t))}
    \norma{\left((F-G)-\div(f-g)\right)(t, y,\cdot)}_{\L{\infty}([-{V}_t,{V}_t])}
    \,\mathrm{d}y  \,,
  \end{eqnarray*}
  since $T \leq T_0$. Consequently
  \begin{equation}
    \label{aab}
    A'(T)
    \leq
    A'(0)
    +
    \kappa^* A(T)
    + 
    \left(
      \frac{e^{\kappa^*_0 T} - 1}{\kappa^*_0} a
      +
      \int_0^T \frac{e^{\kappa^*_0 (T-t)} -1}{\kappa^*_0} b(t)\mathrm{d}t 
      +
      \int_0^T c(t)\mathrm{d}t \right)
    \,.
  \end{equation}
  By a Gronwall type argument, we obtain
  \begin{align*}
    A'(T)
     \leq & 
    e^{\kappa^* T} A'(0)
    +
    \frac{e^{\kappa^*_0 T}-e^{\kappa^* T}}{\kappa^*_0-\kappa^*} \, a
    +
    \int_0^T \frac{e^{\kappa^*_0 (T-t)}-e^{\kappa^* (T-t)}}{\kappa^*_0-\kappa^*}\, b(t)
    \, \mathrm{d}t 
    +
    \int_0^T e^{\kappa^*(T-t)}c(t)\, \mathrm{d}t\, .
  \end{align*}
  Taking $T=T_0$, we finally obtain the result.
\end{proofof}

\begin{proofof}{Proposition \ref{thm:stab2}}
Thanks to Lemma \ref{lem:main}, we can write 
  \begin{equation}\label{eq:aaa2}
    B'(T) \leq B'(0) + \kappa^* \, B(T) + S(T) \,,
  \end{equation}
  where
  \begin{eqnarray}
    \nonumber
    B(T)
    & = &
    \int_0^{T} \int_{B(x_0, R +M(T_0-t))}
    \modulo{v(t,x)-u(t,x)}
    \, \mathrm{d}x \, \mathrm{d}t\, ,
    \\
    \label{eq:table2}
    \kappa^*
    & = &
    \norma{\pt_u F }_{\L{\infty}( \Sigma_{T_0}^{u,v})}+\norma{\pt_u \div(g-f)}_{\L\infty(\Sigma_{T_0}^{u,v})}\,,
    \\
    \nonumber
    S(T)
    & = &
    \sup_{t\in [0,T_0]} \,\tv\left(u(t) \right)\,
    \int_0^T \norma{\pt_u(f-g)(t)}_{\L{\infty}(\mathcal{S}_{T_0}(u)\times[-U_t,U_t]  )}\, \mathrm{d}t
    \\
 \nonumber   & &
    \quad
    +
    \int_0^T \int_{B(x_0,R+M(T_0-t))}\!
    \norma{\left((F-G)-\div(f-g)\right)(t,y,\cdot)}_{\L{\infty}([-{V}_t,{V}_t])}
    \!
    \mathrm{d}y \, \mathrm{d}t .
  \end{eqnarray}
  The bound~(\ref{result}) on $\tv\left(u(t)\right)$ gives:
  \begin{eqnarray*}
    S(T)
    & \leq & 
    \left(e^{\kappa_0^* T}\tv(u_0)+NW_N  \int_0^T e^{\kappa_0^*(T-t)}\int_{\reali^N}
    \norma{\nabla(F-\div  f)(t,x,\cdot)}_{\L\infty([-{U}_t,{U}_t])}
    \, \mathrm{d}x\d{t} \right)\\
    &&\qquad \times\int_0^T \norma{\pt_u(f-g)(t)}_{\L{\infty}(\mathcal{S}_{T_0}(u)\times[-U_t,U_t]  )}\, \mathrm{d}t\\
    &&+\int_0^T \int_{B(x_0,R+M(T_0-t))}
    \norma{\left((F-G)-\div(f-g)\right)(t, y,\cdot)}_{\L{\infty}([-{V}_t,{V}_t])}
    \,\mathrm{d}y  \,,\d{t}
  \end{eqnarray*}
  where $\kappa^*_0$ is defined in~(\ref{eq:kappao}). 
Let us denote
\begin{align*}
a=& \tv(u_0)\,,\\
b(t)=& NW_N   e^{-\kappa_0^*t}\int_{\reali^N}
    \norma{\nabla(F-\div  f)(t,x,\cdot)}_{\L\infty([-{U}_t,{U}_t])}
    \, \mathrm{d}x\,,\\
    c(t)=&\norma{\pt_u(f-g)(t)}_{\L{\infty}(\mathcal{S}_{T_0}(u)\times[-U_t,U_t]  )}\,,\\
    d(t)=& \int_{B(x_0,R+M(T_0-t))}
    \norma{\left((F-G)-\div(f-g)\right)(t, y,\cdot)}_{\L{\infty}([-{V}_t,{V}_t])}
    \,\mathrm{d}y \,.
\end{align*}  
Then we have 
\[
A'(T)\leq A'(0)+\kappa^* A(T)+e^{\kappa_0^* T}\left( a+\int_0^T b(t)\d{t}\right)\int_0^T c(t)\d{t}+\int_0^T d(t)\d{t}\,.
\]
  Consequently, by a Gronwall type argument, we obtain
  \begin{align*}
    B'(T)
     \leq & 
    e^{\kappa^* T} B'(0)
    +\frac{\kappa_0^*e^{\kappa_0^* T}-\kappa^*e^{\kappa^* T}}{\kappa_0^*-\kappa^*} \left(a+\int_0^T b(t)\d{t}\right)\int_0^T c(t)\d{t} +e^{\kappa^* T}\int_0^T d(t)\d{t}\,.
  \end{align*}
  Taking $T=T_0$, we finally obtain the result.
\end{proofof}

\section{Technical tools}\label{sec:tech}
We give below a lemma that was used in the previous proof. 
Let us recall from \cite{ColomboMercierRosini} the following useful technical results:
\begin{lemma}

  \label{lem:mu}
  Fix a function $\mu_1 \in \Cc\infty(\rpic;\rpic)$ with
  \begin{equation}
    \label{eq:mu}
    \mathrm{Supp}(\mu_1) \subseteq \left[0, 1 \right[
    ,\quad
    \int_{\rpis}
    r^{N-1} \mu_1(r) \, \mathrm{d}r = \frac{1}{N \omega_N}
    ,\quad
    \mu_1' \leq 0
    ,\quad
    \mu_1^{(n)}(0) = 0 \mbox{ for } n\geq 1.
  \end{equation}
  Define
  \begin{equation}
    \label{eq:Mu}
    \mu(x) = \frac{1}{\gd^N} \, \mu_1 \left(
      \frac{\norma{x}}{\gd} \right) \,.
  \end{equation}
  Then, recalling that $\omega_0 = 1$,
  \begin{eqnarray}
    \label{eq:mu1}
    \int_{\reali^N} \mu(x) \,\mathrm{d} x
    & = &
    1 \, ,
    \\
    \label{eq:mu2}
    \int_{\reali^N}
    \modulo{x_1} \, \mu_1 \left( \norma{x}\right) \, \mathrm{d}x
    & = &
    \frac{2}{N} \, \frac{\omega_{N-1}}{\omega_N} \,
    \int_{\reali^N}
    \norma{x} \, \mu_1 \left( \norma{x}\right) \, \mathrm{d}x\, ,
    \\
    \label{eq:mu3}
    \int_{\reali^N}
    \norma{x} \, \norma{\nabla \mu (x)} \, \mathrm{d}x
    & = &
    - \int_{\reali^N}
    \norma{x} \, \mu_1' \left( \norma{x}\right) \, \mathrm{d}x
    \;\, = \;\,
    N\,  ,
    \\
    \label{eq:mu4}
    \int_{\reali^N}
    \norma{x}^2 \, \mu_1' \left( \norma{x} \right) \, \mathrm{d}x
    & = &
    - (N+1) \,
    \int_{\reali^N}
    \norma{x} \, \mu_1 \left( \norma{x}\right) \, \mathrm{d}x \,.
  \end{eqnarray}
\end{lemma}

\begin{lemma}
  \label{lem:estI}
  Let $I$ be defined as in~(\ref{eq:I}). Then,
  \begin{eqnarray*}
    \limsup_{\epsilon\to0} I 
    & \leq &
    \int_{\norma{x-x_0} \le R+MT_0+\theta}
    \modulo{u(0,x)-v(0,x)} \, \mathrm{d}x
    \\
    & &
    -
    \int_{\norma{x-x_0}\le R +M(T_0-T)} 
    \modulo{u(T,x)-v(T,x)} \, \mathrm{d}x 
    + 
    2 \sup_{\tau\in\{0,T\}} \tv \left(u(\tau) \right) \gd
    \\
    & &      
    +
    2 \sup_{t\in \{0,T\} \atop s\in\left]t,t+\pd\right[}
    \int_{\norma{y-x_0}\le R +\gd+M(T_0-t)+\theta}
    \modulo{u(t,y)-u(s,y)} \, \mathrm{d}y \,.
  \end{eqnarray*}
\end{lemma}

\begin{proof}
See \cite[Lemma 5.2]{ColomboMercierRosini}.
\end{proof}

\small{

 \bibliography{ref}

 \bibliographystyle{abbrv} }

\end{document}